\documentclass{amsart}
\usepackage{amssymb}
\usepackage{graphicx}
\usepackage{bm}

\allowdisplaybreaks[1]

\newcommand{\C}{\mathbb C}
\newcommand{\D}{\mathbb D}

\newcommand{\N}{\mathbb N}
\newcommand{\R}{\mathbb R}

\renewcommand{\H}{\mathbb H}
\renewcommand{\P}{\mathbb P}

\newcommand{\Z}{\mathbb Z}

\newcommand{\cF}{\mathcal F}

\newcommand{\cM}{\mathcal M}

\newcommand{\re}{\mathrm{Re}}
\newcommand{\im}{\mathrm{Im}}

\newcommand{\pa}{\partial}
\newcommand{\ex}{\mathbf{e}}

\renewcommand{\a}{\alpha}
\renewcommand{\b}{\beta}
\newcommand{\g}{\gamma}
\renewcommand{\d}{\delta}
\newcommand{\De}{\mathit{\Delta}}
\renewcommand{\l}{\lambda}
\renewcommand{\L}{\Lambda}

\newcommand{\h}{\vartheta}

\newcommand{\e}{\varepsilon}
\newcommand{\f}{\varphi}

\newcommand{\G}{\mathit{\Gamma}}

\newcommand{\W}{\mathit{\Omega}}

\renewcommand{\i}{\mathbf{i}}

\newcommand{\la}{\langle}
\newcommand{\ra}{\rangle}
\newcommand{\tr}{\;^t}
\newcommand{\sint}{\begin{matrix} \int\end{matrix}}

\newcommand{\pr}{\mathrm{pr}}

\newcommand{\comment}[1]{}

\newcommand{\ds}[1]{\displaystyle{#1}}

\title[Schwarz's map for $F_2$]
{Schwarz's map for Appell's second hypergeometric system 
with quarter integer parameters}
\author{Keiji Matsumoto}
\address[Matsumoto]{
Department of Mathematics,
Faculty of Science,
Hokkaido University,
Sapporo 060-0810, Japan
}
\email{matsu@math.sci.hokudai.ac.jp}
\author{Shohei Osafune}
\address[Osafune]
{DAITEC Co., Ltd., 
Kita-4 Higashi-1 2-3, Chuo-ku,
Sapporo 060-0034, Japan
}

\author{Tomohide Terasoma}
\address[Terasoma]{
Faculty of Science and Engineering,
Hosei University,
Koganei, Tokyo 184-8584, Japan
}
\email{terasoma@hosei.ac.jp}

\date{\today}

\theoremstyle{plain} 
\newtheorem{theorem}{\indent\sc Theorem}[section]
\newtheorem{lemma}[theorem]{\indent\sc Lemma}
\newtheorem{cor}[theorem]{\indent\sc Corollary}
\newtheorem{proposition}[theorem]{\indent\sc Proposition}

\theoremstyle{definition} 
\newtheorem{definition}[theorem]{\indent\sc Definition}
\newtheorem{fact}[theorem]{\indent\sc Fact}
\newtheorem{remark}[theorem]{\indent\sc Remark}

\numberwithin{equation}{section}

\begin{document}
\maketitle
\begin{abstract}
We study Schwarz's map for Appell's second system $\cF_2$ of hypergeometric
differential equations in two variables with parameters
$a=c_1=c_2=\frac{1}{2}$, $b_1=b_2=\frac{1}{4}$. 
By using theta functions with characteristics, 
we give a defining equation of an analytic set in $\C^2\times \H$ 
of its image, and express its inverse.  
\end{abstract}

\tableofcontents 
\section{Introduction}
\label{sec:Intro}

The hypergeometric differential equation
$\cF(a,b,c)$  
with complex parameters $a,b,c$ is defined by  
\begin{equation}
\label{eq:HGDE}
\cF(a,b,c):   
\big[\frac{d }{d  x}(c-1+x\frac{d }{d  x})
    -(a+x\frac{d }{d  x})(b+x\frac{d }{d  x})\big] f(x)=0.
  \end{equation}
It is satisfied by the hypergeometric series
\begin{equation}
\label{eq:HGS}  
F(a,b,c;x)=\sum_{n=0}^\infty \frac{(a,n)(b,n)}{(c,n)n!} x^{n}
\quad
\Big(
\begin{array}{l}
|x|<1,\ c\notin -\N_0=\{0,-1,-2,\dots\}\\
(a,n)=a(a+1)\cdots(a+n-1)=\frac{\G(a+n)}{\G(a)}
\end{array}
     \Big),
\end{equation}
and by Euler type integrals 
\begin{equation}
\label{eq:Int-rep-HGS}  
\int_\sigma t^{b}(1-t)^{c-b}(1-tx)^{-a}\frac{dt}{t(1-t)},
\end{equation}
where $\sigma$ are $1$-chains with some kinds of boundary conditions.
Since $\cF(a,b,c)$ is a linear equation of rank $2$ with 
singular points $x=0,1,\infty$, 
the space of solutions to $\cF(a,b,c)$ 
on a small neighborhood $V$ of a point $x\in \C-\{0,1\}$ 
is a $2$-dimensional vector space. 
Schwarz's map is defined by the analytic continuation of the 
ratio of linearly independent solutions to $\cF(a,b,c)$ on $V$ to 
the space $\C-\{0,1\}$ (or the universal covering of $\C-\{0,1\}$). 
Though it is multi-valued in general, 
if $a,b,c$ satisfy $1/|1-c|,1/|c-a-b|,1/|a-b|\in \{2,3,4,\dots,\infty\}$, 
then its inverse is single valued, moreover if they satisfy 
$|1-c|+|c-a-b|+|a-b|<1$, then its image is isomorphic to an open dense subset 
in the upper-half plane $\H$. 
There are several detailed studies of Schwarz's map and its inverse  
for some special parameters. 

In this paper, we consider Appell's second system 
$\cF_2(a,b_1,b_2,c_1,c_2)$ of hypergeometric differential 
equations in two variables with parameters $a,b_1,b_2,c_1,c_2$, 
see \eqref{eq:Appell-F2} for their explicit forms.
It is a regular holonomic system of rank $4$ with regular locus 
$$X=\{(x_1,x_2)\in \C^2\mid x_1(1-x_1)x_2(1-x_2)(1-x_1-x_2)\ne0\}.$$
We specialize parameters as 
$$a=c_1=c_2=\frac{1}{2},\quad  b_1=b_2=\frac{1}{4},$$ 
and set 
$$\cF_2=\cF_2(\frac{1}{2},\frac{1}{4},\frac{1}{4},\frac{1}{2},\frac{1}{2}).$$
By taking the ratio of a fundamental system of solutions to 
$\cF_2$ around a point $(\dot x_1,\dot x_2)$ in the regular locus $X$, 
we have a map from a neighborhood $U$ of $(\dot x_1,\dot x_2)$ 
to the complex projective space $\P^3$. 
Our main object in this paper is 
Schwarz's map $\mathbf{f}$ for $\cF_2$ defined by the analytic 
continuation of this map to $X$ (or to the universal covering of $X$). 
Since our specialized parameters satisfy $a=c_1=c_2$,  
results in \cite[\S1.3]{MSTY} imply that $\cF_2$ is reducible, and that 
Euler type double integrals \eqref{eq:F2-Int-rep} for solutions to $\cF_2$ 
reduce to products of path integrals and elementary functions as in 
Proposition \ref{prop:a=c1=c2} and \eqref{eq:periods}, where 
a variable $$z=\frac{1-x_1-x_2}{(1-x_1)(1-x_2)}$$
is introduced from the variables $x_1,x_2$ in $\cF_2$. 
As a consequence of these, we have classical Schwarz's map for 
$\cF(\frac{1}{4},\frac{1}{4},1)$ 
from our Schwarz's map $\mathbf{f}$ for $\cF_2$ 
by restricting variables to $z$ and taking a $2$-dimensional subspace
associated with solutions to $\cF(\frac{1}{4},\frac{1}{4},1)$. 
And the integrals \eqref{eq:periods} suggest us to 
introduce a family of algebraic curves:
$$\bigcup_{z\in \C-\{0,1\}} C_z,\quad C_z: w^4=v^3(1-v)(1-vz).$$
Since every member $C_z$ of this family can be regarded as a 
cyclic quadruple covering of the complex projective line $\P^1$ 
with covering transformation 
$$\sigma:C_z\ni (v,w)\mapsto (v,\i w)\in C_z\quad (\i=\sqrt{-1}),$$
it is of genus $3$ by Hurwitz's formula. 
We regard \eqref{eq:periods} as 
path integrals of a holomorphic $1$-form $\eta_1=dv/w$ on $C_z$,  
which belongs to the $(-1)$-eigenspace $H^0_{-\sigma^2}(C_z,\W)$ 
in the space $H^0(C_z,\W)$ of holomorphic $1$-forms on $C_z$ under the 
action of the involution $(\sigma^2)^*$.
In studies in \cite{MSTY} and \cite{MT}, we consider 
a family of algebraic curves of genus $2$ and 
the Abel-Jacobi map from each member of this family to its Jacobi variety, 
which is a principally polarized abelian variety.  
In this paper,  instead of the Jacobi variety of $C_z$, 
we consider the Prym variety 
$$H_{-\sigma^2}^0(C_z,\W)^*/H_1^{-\sigma^2}(C_z,\Z)$$
of $C_z$ with respect to the involution $\sigma^2$,
where $H_1^{-\sigma^2}(C_z,\Z)$ is the $(-1)$-eigenspace in  
$H_1(C_z,\Z)$ under the action of the involution $\sigma^2$, 
and $H_{-\sigma^2}^0(C_z,\W)^*$ is regarded as a subspace of $H_1(C_z,\C)$
through the pairing between $H_{-\sigma^2}^0(C_z,\W)$ and $H_1(C_z,\C)$
defined by the integral.
It is a $2$-dimensional abelian variety with non-principal polarization. 
Though the Prym variety itself does not seem to have a simple structure, 
there exists a subgroup $\L$ of index $2$ in $H_1^{-\sigma^2}(C_z,\Z)$ 
such that 
$$J_\L(C)=H_{-\sigma^2}^0(C_z,\W)^*/\L$$ admits a decomposition 
$T_\tau\oplus T_\tau$, where $T_\tau$ is the complex torus $\C/(\Z\tau+\Z)$
defined by an element $\tau \in \H$.
We define the Abel-Jacobi $\L$-map by 
$$\jmath_\L:C_z \ni P \mapsto \big(
H_{-\sigma^2}^0(C_z,\W)\ni \f \mapsto (1-\sigma^2)\int_{P_0}^P \f \in 
\C\big)\in J_\L(C)
\simeq T_\tau\oplus T_\tau,$$
where $P_0=(0,0)\in C_z$ is fixed as an initial point.
Though the integral depends on paths connecting $P_0$ to $P$, 
it uniquely gives an element of $J_\L(C)=H_{-\sigma^2}^0(C_z,\W)^*/\L$
since $\L$ includes $(1-\sigma^2)H_1(C_z,\Z)$. 
Proposition \ref{prop:inje-jmath-L} states 
that the map $\jmath_\L$ is of degree $1$, and that 
the compound 
of $\jmath_\L$ and the first (or second)  projection 
$T_\tau\oplus T_\tau\to T_\tau$ is of degree $2$.
By relating  Schwarz's map $\mathbf{f}$ for $\cF_2$ to 
the Abel-Jacobi $\L$-map $\jmath_\L$, we determine the image of $\mathbf{f}$
in \eqref{eq:theta-function} 
as an analytic set in $\C^2\times \H$ in terms of theta functions 
$\h_{k,\ell}(y,\tau)$ associated with $T_\tau$. 
To obtain its defining equation, 
we represent a meromorphic function $h(v)=\dfrac{v(v-1)}{v-1/z}$ on $C_z$ 
in terms of $\h_{k,\ell}(y,\tau)$, the Abel-Jacobi $\L$-map $\jmath_\L$, 
and the two projections $T_\tau\oplus T_\tau\to T_\tau$.
We have two representations, and their coincidence implies 
the defining equation. 
By considering branched values of this meromorphic function, we have 
$$z=\frac{4\h_{01}(0,\tau)^4\h_{10}(0,\tau)^4}{\h_{00}(0,\tau)^8},
$$
which is an expression of the inverse of 
Schwarz's map for $\cF(\frac{1}{4},\frac{1}{4},1)$.
To express the inverse of Schwarz's map $\mathbf{f}$ for $\cF_2$, 
we represent a meromorphic function $1-v$ on $C_z$ 
in terms of $\h_{k,\ell}(y,\tau)$ and the Abel-Jacobi $\L$-map $\jmath_\L$
in Theorem \ref{th:(1-v)-theta-exp}. 
By this representation, we can easily obtain an 
expression of $\mathbf{f}^{-1}$, see Theorem \ref{th:inv-period}.

We study the monodromy representation of $\cF_2$ in Section \ref{sec:monod}. 
It is known that the fundamental group of $X$ is generated by five loops.  
The circuit matrices of $\cF_2(a,b,c)$ along these loops 
are given in \cite[Corollary 3.12]{MSTY}, which is valid 
even in the case $a=c_1=c_2$. 
We give their explicit forms after specializing parameters and 
taking conjugates compatible with the Abel-Jacobi $\L$-map $\jmath_\L$.
We characterize the group generated by these five matrices 
as a congruence subgroup of $GL_4(\Z[\i])$ in Theorem \ref{th:str-monod}. 

In Appendix, we investigate the structure of the field $\C(v,w)$ of 
rational functions on $C_z$. As pointed out in previous, 
to express the inverse $\mathbf{f}^{-1}$, the key is 
the representation of the meromorphic function 
$1-v$ on $C_z$ in Theorem \ref{th:(1-v)-theta-exp}.
Though our proof of this theorem is based not on Appendix but  
on elliptic function theory,  
we find the representation through observations in Appendix. 
We give an alternative proof of Theorem \ref{th:(1-v)-theta-exp} 
based on Kummer theory for fields of rational functions on 
curves given by quotients of $C_z$ by several groups 
generated by $\sigma$ and an involution $\iota$ associated with 
the complex torus $T_\tau$.
We give generators of a series of quadratic extensions from  
$\C(v)^\iota$ to $\C(v,w)$,   
where $\C(v)^\iota$ denotes the fixed field of the rational function field 
$\C(v)$ under the involution $\iota$, and $[\C(v,w):\C(v)^\iota]=8$.

\section{Appell's second system $\cF_2(a,b,c)$}
\label{sec:Appell-F2}
In this section, by referring \cite{MSTY}, \cite{MT} and \cite{Y1},
we introduce some facts about Appell's second system $\cF_2(a,b_1,b_2,c_1,c_2)$
of hypergeometric differential equations,
which is a generalization of the hypergeometric differential equation
$\cF(a,b,c)$. 
Appell's second hypergeometric series $F_2(a,b,c;x)$ is
defined by
$$F_2(a,b_1,b_2,c_1,c_2;x_1,x_2)=
\sum_{n_1=0}^\infty\sum_{n_2=0}^\infty
\frac{(a,n_1+n_2)(b_1,n_1)(b_2,n_2)}
{(c_1,n_1)(c_2,n_2)n_1!n_2!} x_1^{n_1}x_2^{n_2},
$$
where $x_1,x_2$ are main variables,
$a,b_1,b_2,c_1,c_2$ are complex parameters with
$c_1,c_2\notin -\N_0$.
It converges absolutely and uniformly on any compact set in the domain
$$\D=\{x=(x_1,x_2)\in \C^2\mid |x_1|+|x_2|<1\},$$
and satisfies hypergeometric differential equations 
\begin{equation}
\label{eq:Appell-F2}
\begin{array}{c}
\big[\dfrac{\pa}{\pa x_1}(c_1-1+x_1\dfrac{\pa}{\pa x_1})
    -(a+x_1\dfrac{\pa}{\pa x_1}+x_2\dfrac{\pa}{\pa x_2})
     (b_1+x_1\dfrac{\pa}{\pa x_1})\big] f(x)=0,\\[3mm]
\big[\dfrac{\pa}{\pa x_2}(c_2-1+x_2\dfrac{\pa}{\pa x_2})
    -(a+x_1\dfrac{\pa}{\pa x_1}+x_2\dfrac{\pa}{\pa x_2})
    (b_2+x_2\dfrac{\pa}{\pa x_2})\big]f(x)=0.
\end{array}
\end{equation}
The system generated by these partial differential equations is called 
Appell's second system $\cF_2(a,b,c)=\cF_2(a,b_1,b_2,c_1,c_2)$ of 
hypergeometric differential equations. 
It is a regular holonomic system of rank $4$
with singular locus
$$S=\{x\in \C^2\mid x_1x_2(1-x_1)(1-x_2)(1-x_1-x_2)=0\}\cup L_\infty,$$
where $L_\infty$ is the line at infinity in the projective plane $\P^2$.
\begin{fact}[{\cite[\S6.4]{Y1}}]
\label{fact:iintrep}
Euler type integrals
\begin{equation}
  \label{eq:F2-Int-rep}
  \iint_D u(t,x)dt, \quad 
  u(t,x)=  t_1^{b_1-1}(1-t_1)^{c_1-b_1-1}
   t_2^{b_2-1}(1-t_2)^{c_2-b_2-1}(1-t_1x_1-t_2x_2)^{-a},
 \end{equation}  
satisfy Appell's second system $\cF_2(a,b_1,b_2,c_1,c_2)$,  
where  $dt=dt_1\wedge dt_2$, and  $2$-chains $D$ 
satisfy some kind of boundary conditions. 
\end{fact}

\begin{fact}[\cite{Bo}]
\label{fact:irreducible}
Appell's second system $\cF_2(a,b_1,b_2,c_1,c_2)$ is reducible
if and only if at least one of
\begin{equation}
\label{eq:reducible}  
  a,\ b_1,\ b_2,\ b_1-c_1,\ b_2-c_2,\ a-c_1,\ a-c_2,\ a-c_1-c_2
\end{equation}  
is an integer.
\end{fact}

We take a base point $\dot x=(\e,\e)$ of the complement
$$X=\{x\in \C^2\mid x_1(1-x_1)x_2(1-x_2)(1-x_1-x_2)\ne 0\}$$
of the singular locus $S$ of $\cF_2(a,b,c)$, where
$\e$ is a sufficiently small positive real number.
We give a basis of the space of local solutions to $\cF_2(a,b,c)$
on a small neighborhood $\dot U$ of $\dot x$.
\begin{fact}
\label{fact:basis}
Suppose that none of (\ref{eq:reducible}) is an integer. Then  
the space of local solutions to $\cF_2(a,b,c)$ on
$\dot U$ is spanned by the four integrals \ref{eq:F2-Int-rep} for 
$D=D_1,\dots,D_4$, which are indicated in Figure \ref{fig:2-cycles}.
\begin{figure}
\includegraphics[width=8cm]{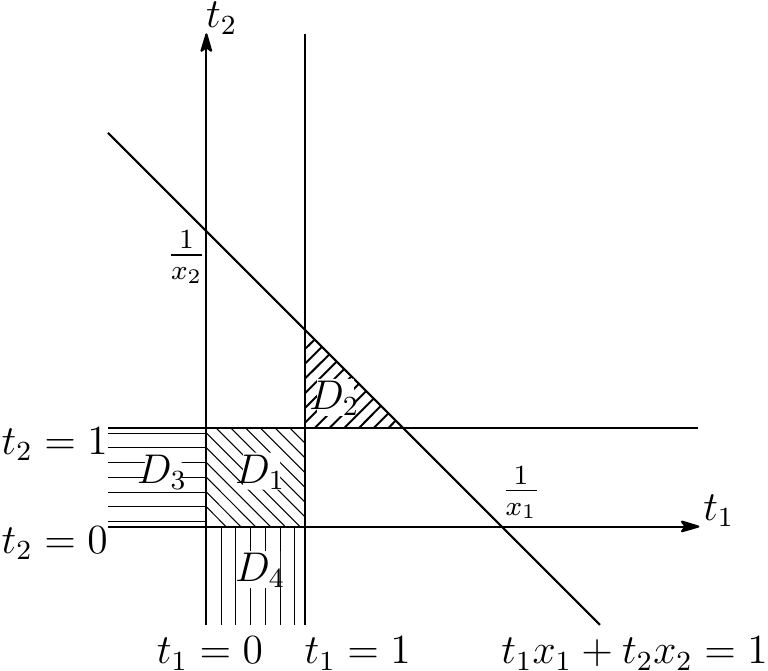}
\caption{Areas of integrals}
\label{fig:2-cycles}
\end{figure}
Here the branches of the integrand on these area for $x\in \dot U\cap \R^2$ 
are assigned by the arguments of $t_1,1-t_1,t_2,1-t_2,1-t_1x_1-t_2x_2$ as 
in Table \ref{tab:arg-t}, 
\begin{table}[htb]
  \centering
  \begin{tabular}{|c|c|c|c|c|c|}
    \hline
          & $t_1$ & $1-t_1$ &$t_2$ & $1-t_2$ & $1-t_1x_1-t_2x_2$\\
    \hline
    $D_1$ &  $0$  & $0$     & $0$  &  $0$   &  $0$   \\    
    $D_2$ &  $0$  & $\pi$   & $0$  &  $\pi$ &  $0$   \\
    $D_3$ & $-\pi$& $0$     & $0$  &  $0$   &  $0$   \\
    $D_4$ &  $0$  & $0$     &$-\pi$&  $0$   &  $0$   \\
    \hline
  \end{tabular}\\[2mm]
  \caption{Branches of the integrand}
  \label{tab:arg-t}
\end{table}
and the boundaries of $D_1,\dots,D_4$ are eliminated as in \cite[\S3.2.4]{AK} 
under our non-integral assumption.
\end{fact}
\begin{remark}
  \label{rem:reduced-case}
  The basis in Fact \ref{fact:basis}  is valid even in cases
$a-c_1\in \N_0$ or $a-c_2\in \N_0$. 
\end{remark}

When the parameters $a,c_1,c_2$ of $\cF_2(a,b,c)$ satisfy $a=c_1=c_2$,
the structure of its fundamental system of solutions is studied in
\cite[Theorem 2.4]{MSTY}, which suggests us that 
the fundamental solutions to $\cF_2(a,b,c)$
in Fact \ref{fact:basis} can be expressed in terms of the hypergeometric series
$F(a,b,c;x)$ and Appell's first hypergeometric series 
$$F_1(a,b_1,b_2,c;x_1,x_2)=\sum_{n_1=0}^\infty\sum_{n_2=0}^\infty
\frac{(a,n_1+n_2)(b_1,n_1)(b_2,n_2)}{(c,n_1+n_2)n_1!n_2!}x^{n_1}x^{n_2},$$
admitting an Euler type integral representation 
$$
\frac{1}{B(a,c-a)}\int_0^1 t_1^{a-1}(1-t_1)^{c-a-1}(1-t_1x_1)^{-b_1}
(1-t_1x_2)^{-b_2}dt_1.
$$
In fact, by following results in \cite[\S1.3.2]{MSTY} and \cite[\S2.4]{MT}, 
we have the expressions of $\iint_{D_i} u(t,x)dt$ $(i=1,\dots,4)$.
\begin{proposition}
\label{prop:a=c1=c2}
Suppose that
$$a,b_1,b_2,b_1-c_1,b_2-c_2,a-c_1-c_2\notin \Z.
$$
If the parameters of Appell's second system $\cF_2(a,b,c)$ satisfy 
$a=c_1=c_2$ then the solutions $\iint_{D_i} u(t,x)dt$ in 
Fact \ref{fact:basis} reduce to 
\begin{align*}
  \iint_{D_1} u(t,x)dt =&B(b_2,a-b_2)(1-x)^{-b}
                         \int_{-\infty}^0 (-s_1)^{b_1-1}(1-s_1)^{b_2-a}
                         (1-s_1z)^{-b_2}ds_1\\
  =&B(b_1,a-b_1)B(b_2,a-b_2)
     (1-x_1)^{-b_1}(1-x_2)^{-b_2}
     F(b_1,b_2,a;1-z),\\
  \iint_{D_2} u(t,x)dt =
&  \ex(a\!-\!\frac{b_1\!+\!b_2}{2})B(a\!-\!b_2,1\!-\!a)(1-x)^{-b}\!
   \int_{1/z}^\infty s_1^{b_1-1}(s_1-1)^{b_2-a}(s_1z-1)^{-b_2}ds_1\\
  =&\ex(a\!-\!\frac{b_1\!+\!b_2}{2})
 \frac{\G(a-b_2)\G(1-a)\G(a-b_1)}{\G(a-b_1-b_2+1)}(1-x_1)^{b_2-a}(1-x_2)^{b_1-a}
      \\
  &\cdot (1-x_1-x_2)^{a-b_1-b_2}   F(a-b_1,a-b_2,a-b_1-b_2+1;z),\\
\iint_{D_3} u(t,x)dt
  =&B(b_2,a-b_2)\ex(\frac{1\!-\!b_1}{2})(1-x)^{-b}
   \cdot\int_0^{1-x_1} s_1^{b_1-1}(1-s_1)^{b_2-a}(1-s_1z)^{-b_2}ds_1\\
  =&\frac{B(b_2,a\!-\!b_2)}{b_1}\ex(\frac{1\!-\!b_1}{2})(1-x_2)^{-b_2}
  F_1\big(b_1,a\!-\!b_2,b_2,b_1\!+\!1;1-x_1,\frac{1\!-\!x_1\!-\!x_2}{1-x_2}\big),\\
  \iint_{D_4} u(t,x)dt
=&B(b_1,a-b_1)\ex(\frac{1\!-\!b_2}{2})(1-x)^{-b}
\cdot\int_0^{1-x_2} s_1^{b_2-1}(1-s_1)^{b_1-a}(1-s_1z)^{-b_1}ds_1\\
=&\frac{B(b_1,a\!-\!b_1)}{b_2}\ex(\frac{1\!-\!b_2}{2})(1-x_1)^{-b_1}
F_1\big(b_2,b_1,a\!-\!b_1,b_2\!+\!1;\frac{1\!-\!x_1\!-\!x_2}{1-x_1},1-x_2\big)
  \\
  =&B(b_1,a-b_1)\ex(\frac{1-b_2}{2})x_1^{1-a}x_2^{1-a}
     (1-x_1)^{a-b_1-1}(1-x_2)^{a-b_2-1}
     \\
 &\cdot\int^1_{1-x_1} s_2^{b_1-a}(1-s_2)^{b_2-1}(1-s_2z)^{a-b_2-1}ds_2,\\
\end{align*}
where we set 
$$(1-x)^{-b}=(1-x_1)^{-b_1}(1-x_2)^{-b_2},$$
and \begin{equation}\label{eq:z}
z=\frac{1-x_1-x_2}{(1-x_1)(1-x_2)}, 
\end{equation}
belonging  to the open interval $(0,1)$ for $(x_1,x_2)\in \dot U\cap \R^2$.  
We regard these formulas as the continuation of
the equalities for $(x_1,x_2)\in \dot U\cap \R^2$,     
and each integrand supposed to be positive real on 
the integration intervals in this restricted case. 
When improper integrals do not converge,  
integration intervals are regraded as regularized ones
in sense of \cite[\S2.3.2]{AK}
\end{proposition}

We can make the analytic continuation of any local solution $f(x)$ to 
$\cF_2(a,b,c)$ on $\dot U$ along any path in $X$. 
In particular, a loop $\rho$ in $X$ with terminal $\dot x$ 
gives rise to a linear transformation 
$$\cM_\rho:f(x)\mapsto \rho_*(f)(x)$$ 
of the space of local solutions to $\cF_2(a,b,c)$ on $\dot U$.
Here  $\rho_*(f)(x)$ denotes the analytic continuation of $f(x)$ along $\rho$, 
and $\cM_\rho$ is called the circuit transformation along $\rho$. 
The correspondence $\rho\mapsto \cM_\rho$ induces a homomorphism 
from the fundamental group  $\pi_1(X,\dot x)$ to 
the general linear group of the space of local solutions to 
$\cF_2(a,b,c)$ on $\dot U$.
This homomorphism is called the monodromy representation of $\cF_2(a,b,c)$.
Its detailed structure is studied in \cite[\S3]{MSTY}.   
As one of its results, for loops $\rho_1,\dots,\rho_5$ generating 
$\pi_1(X,\cdot x)$, 
\cite[Corollary 3.12]{MSTY}  gives 
the representation matrices $M_i^\mu$ $(i=1,\dots,5)$ of 
the circuit transformations along $\rho_i$ 
with respect to the basis in Fact \ref{fact:basis}, 
where $\mu$ consists of parameters 
$$\mu=(\mu_1,\dots,\mu_5)=(e^{2\pi\i(b_1-1)},
e^{2\pi\i(c_1- b_1-1)}, e^{2\pi\i(b_2-1)},e^{2\pi\i(c_2- b_1-1)}, e^{2\pi\i(-a)}).
$$

\section{Schwarz's map for $\cF_2$}
\label{sec:Schwarz's map}
Hereafter we fix the parameters of $\cF_2(a,b,c)$ as  
$$(a,b,c)=(a,b_1,b_2,c_1,c_2)
=\big(\frac{1}{2},\frac{1}{4},\frac{1}{4},\frac{1}{2},\frac{1}{2}\big)
$$
and set 
\begin{equation}
  \label{eq:fixed-parameters}
  \cF_2
=\cF_2\big(\frac{1}{2},\frac{1}{4},\frac{1}{4},\frac{1}{2},\frac{1}{2}\big).
\end{equation}
This system is reducible by Fact \ref{fact:irreducible},
and the integrals
$$f_i(x)=\iint_{D_i}u(t,x)dt=\iint_{D_i}
\frac{dt_1\wedge dt_2}{\sqrt[4]{t_1^3(1-t_1)^3t_2^3(1-t_2)^3
    (1-t_1x_1-t_2x_2)^2}},
$$
in Fact \ref{fact:basis} form a fundamental system of $\cF$ around
$\dot x$ by Remark \ref{rem:reduced-case}. 
\begin{definition}
\label{def:Schwarz-map}
We define Schwarz's map $\mathbf{f}$ for $\cF_2$ by the analytic continuation of
the map
\begin{equation}
  \label{eq:Schwarz-map}
\mathbf{f}: \dot U\ni (x_1,x_2)\mapsto \tr (f_1(x),\dots,f_4(x))\in \P^3
\end{equation}  
to the space
$X=\{(x_1,x_2)\in \C^2\mid x_1(1-x_1)x_2(1-x_2)(1-x_1-x_2)\ne0 \}$.
We regard it as a multi-valued map on $X$ with 
the projective monodromy of $\cF_2$, or as 
a single-valued  map from the universal covering $\widetilde{X}$ of $X$.  
\end{definition}

By Proposition \ref{prop:a=c1=c2} and \eqref{eq:Int-rep-HGS},  
the ratio of $f_1(x)$ and $f_2(x)$ is equal
to that of a fundamental system of solutions to
the hypergeometric equation $\cF(\frac{1}{4}, \frac{1}{4},1)$ 
for an independent variable $z=\dfrac{1-x_1-x_2}{(1-x_1)(1-x_2)}$,
our map $\mathbf{f}$ can be regarded as an extension of classical Schwarz's map
for $\cF(\frac{1}{4}, \frac{1}{4},1)$.
Its image is isomorphic to the upper half space $\H$, and
its monodromy group is isomorphic to the triangle group
$\De(2,\infty,\infty)$
since the difference of characteristic exponents
of $\cF(\frac{1}{4}, \frac{1}{4},1)$ at $0,1,\infty$ are
$0$, $\frac{1}{2}$,  $0$, respectively.

A fundamental domain of $\De(2,\infty,\infty)$ in $\H$ consists of 
two copies of Schwarz's triangles with angles $\frac{\pi}{2},0,0$,
see Figure \ref{fig:FD-igusa}.
\begin{figure}[htb]
  \centering
\includegraphics[width=12cm]{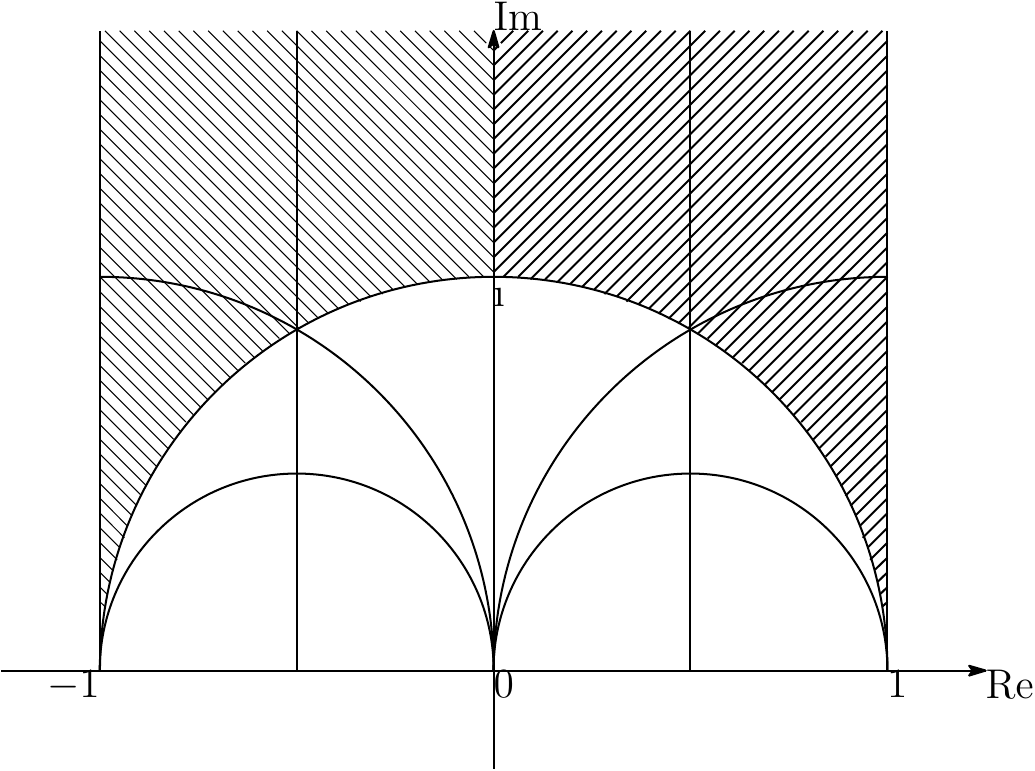}
  \caption{A fundamental domain of the triangle group $\De(2,\infty, \infty)$}
  \label{fig:FD-igusa}
\end{figure}
We identify one of its edge $\{\tau\in \H\mid \re(\tau)=-1\}$ with
$\{\tau\in \H\mid \re(\tau)=1\}$, and
$\{\tau\in \H\mid |\tau|=1, \re(\tau)\le 0\}$ with
$\{\tau\in \H\mid |\tau|=1, \re(\tau)\ge 0\}$ by
Schwarz's reflection principle.
These identifications are realized by 
linear fractional transformations 
$$
\begin{pmatrix}
  1 & 2 \\ 0 & 1 
\end{pmatrix}\cdot \tau=\tau+2,\quad 
\begin{pmatrix}
  0 & -1 \\ 1 & 0
\end{pmatrix}\cdot \tau=\frac{-1}{\tau}
.$$
Thus $\De(2,\infty,\infty)$ is isomorphic to the group generated by
$\begin{pmatrix}
  1 & 2 \\ 0 & 1 
\end{pmatrix}$ and $\begin{pmatrix}
  0 & -1 \\ 1 & 0
\end{pmatrix}$ in $PSL_2(\Z)=SL_2(\Z)/\{\pm E_2\}$,
where $E_2$ is the unit matrix of size $2$.
It is known that these two matrices generate the Igusa group 
$$\G_{1,2}=\Big\{
\begin{pmatrix}
  g_{11} & g_{12} \\ g_{21} & g_{22}
\end{pmatrix}\in SL_2(\Z)\;\Big|\;   
\g_{11}g_{12},\g_{21}g_{22} \equiv 0 \bmod 2 
\Big\}
$$
satisfying
$$[SL_2(\Z):\G_{1,2}]=3,\quad [\G_{1,2}:\G(2)]=2,$$
where $\G(2)$ is the principally congruence subgroup of
$SL_2(\Z)$ of level $2$.

\begin{remark}
  \label{rem:K3}
  The double integrals
  $f_i(x)$
can be regarded as integrals of a holomorphic $2$-form
$\imath^*\big(\dfrac{(1-t_1x_1-t_2x_2)dt_1\wedge dt_2}{t_3^3}\big)$ 
on a $K3$-surface defined by a minimal non-singular model of 
an algebraic surface
$$\{(t_1,t_2,t_3)\in \C^3\mid t_3^4=t_1(1-t_1)t_2(1-t_2)(1-t_1x_1-t_2x_2)^2\},
$$
where $\imath$ is its resolution. This $K3$-surface admits an elliptic 
fibration with four singular fibers of type $III$, $III$, 
$III^*$, $III^*$, and its Picard number is $18$ for generic $(x_1,x_2)\in X$.
\end{remark}

\section{Curves $C_z$ of genus $3$}
\label{sec:curve}
By proposition \ref{prop:a=c1=c2}, the double integrals $f_i(x)$ are
expressed as
\begin{align}
\nonumber  
  f_1(x)&=\frac{B(1/4,1/4)}{\sqrt[4]{(1-x_1)(1-x_2)}}
             \int_{-\infty}^0(-v)^{-3/4}(1-v)^{-1/4}(1-vz)^{-1/4}dv,\\
\nonumber
f_2(x)&=\frac{\i B(1/4,1/2)}{\sqrt[4]{(1-x_1)(1-x_2)}}
             \int_{1/z}^\infty v^{-3/4}(v-1)^{-1/4}(vz-1)^{-1/4}dv,\\
 \label{eq:periods}  \\
\nonumber    f_3(x)&=\frac{\ex(3/8)
             B(1/4,1/4)}{\sqrt[4]{(1-x_1)(1-x_2)}}
             \int_{0}^{1-x_1} v^{-3/4}(1-v)^{-1/4}(1-vz)^{-1/4}dv,\\
\nonumber    f_4(x)&
=\frac{\ex(3/8)
             B(1/4,1/4)}{\sqrt[4]{(1-x_1)(1-x_2)}}
             \int_{0}^{1-x_2} v^{-3/4}(1-v)^{-1/4}(1-vz)^{-1/4}dv.
  \end{align}  
  Here note that
  $$B(\frac{1}{4},\frac{1}{2})=\frac{\G(1/4)\G(1/2)}{\G(3/4)}=
  \G(1/4)\sqrt{\pi} \frac{\G(1/4)\sin(\pi/4)}{\pi}=
  \frac{1}{\sqrt{2}}  B(\frac{1}{4},\frac{1}{4}),
  $$
  and recall that $z=\dfrac{1-x_1-x_2}{(1-x_1)(1-x_2)}$.
We introduce a family of curves parametrized by $z\in \C-\{0,1\}$ 
so that these integrals can be regarded as path integrals
of holomorphic $1$-forms on them.
We consider an algebraic curve 
\begin{equation}
\label{eq:curve}
C_z=\{(v,w)\in \C^2\mid w^4=v^3(1-v)(1-vz)\}
\end{equation}
for any $z\in \C-\{0,1\}$,
and define a compact Riemann surface as its non-singular model, 
which is also denoted by $C_z$.  
Through a natural projection
$$\pr:C_z\ni (v,w)\mapsto v\in \P^1,$$
we regard $C_z$ as a quadruple covering of $\P^1$ with four ramification
points 
$$P_0=(0,0),\quad P_1=(1,0),\quad P_{1/z}=(\frac{1}{z},0),\quad
P_\infty=(\infty,\infty).$$
Since the ramification indices of these points are four,
the Euler number $\chi(C_z)$ of the curve $C_z$ is
$$4\times \chi(\P^1)-4\times (4-1)=-4,$$
which means that this curve is of genus $3$.

The group of the covering transformation of $\pr$ is
the cyclic group of order $4$ generated by the automorphism 
$$\sigma:C_z\ni (v,w)\mapsto (v,\i w)\in C_z.$$
It acts on the vector space $H^0(C_z,\W)$ of homomorphic $1$-forms $\eta$ 
on $C_z$ as the pull-back $\sigma^*(\eta)$.
We can easily obtain the following. 
\begin{proposition}
  \label{prop:sigma-action}
The eigenvalues of $\sigma$ acting on $H^0(C_z,\W)$ is 
$\pm\i$ and $-1$.  
Its eigenspaces of eigenvalues $-\i$, $\i$, $-1$ are spanned by
\begin{align*}
  \eta_1&=\dfrac{dv}{w}=\dfrac{dv}{\sqrt[4]{v^3(1-v)(1-vz)}},\\
  \eta_2&=\dfrac{v^2dv}{w^3}=\dfrac{dv}{\sqrt[4]{v(1-v)^3(1-vz)^3}},\\
  \eta_3&=\dfrac{vdv}{w^2}=\dfrac{dv}{\sqrt{v(1-v)(1-vz)}},
\end{align*}        
respectively.
Their zero divisors  are
\begin{equation}
\label{eq:zero-div}
[\eta_1]=2P_1+2P_{1/z},\quad [\eta_2]=2P_0+2P_\infty,\quad
[\eta_3]=P_0+P_1+P_{1/z}+P_\infty.
\end{equation}
\end{proposition}  

Let 
\begin{equation}
\label{eq:(-1)-sp}
H^0_{-\sigma^2}(C_z,\W)=\{\eta \in H^0(C_z,\W)\mid (\sigma^2)^*(\eta)
=-\eta\}
\end{equation}
be the $(-1)$-eigenspace of the action of the involution 
$\sigma^2$ on $H^0(C_z,\W)$. 
Then $H^0_{-\sigma^2}(C_z,\W)$ is spanned by $\eta_1$ and $\eta_2$ by 
Proposition \ref{prop:sigma-action}.

We study the $(\pm1)$-eigenspaces of the action of 
the involution $\sigma^2$ on the first homology group $H_1(C_z,\Z)$
for a fixed element $z$ of the open interval $(0,1)(\subset X)$.
Since $\sigma^2$ has the four fixed points $P_0$, $P_1$, $P_{1/z}$, $P_\infty$,
the curve $C_z$ can be regarded as a double cover of an elliptic curve
ramified at these four points. It is shown in \cite[Chapter V]{F} that 
there exists a basis $\{A_1,A_2,A_3,B_1,B_2,B_3\}$ of $H_1(C_z,\Z)$ such that
\begin{equation}
\label{eq:sp-basis-sigma-action}
\begin{array}{cl}
A_i\cdot A_j=B_i\cdot B_j=0,&  A_i\cdot B_j=-\d_{ij},\\
\sigma^2(A_1)=A_2,& \sigma^2(B_1)=B_2,\\
\sigma^2(A_3)=-A_3,& \sigma^2(B_3)=-B_3,
\end{array}
\end{equation}
where $\cdot$ is the intersection form on $H_1(C_z,\Z)$, and $\d_{ij}$ is
Kronecker's symbol.
In fact, such basis is given as follows.
\begin{proposition}
\label{prop:symplectic-basis}
Let cycles $A_1$,$A_2$,$A_3$,$B_1$,$B_2$,$B_3$ be given by 
\begin{equation}
\label{eq:sp-basis-expression}
\begin{array}{lll}
  A_1= (1-\sigma)\cdot(I_{0,1}+I_{1,1/z}), 
  &A_2= \sigma^2(A_1),
  &A_3= (1+\sigma)(A_1),\\
  B_1= (1-\sigma)\cdot I_{0,1},
  &B_2= \sigma^2(B_1),
  &B_3= (\sigma+\sigma^2)(B_1),
\end{array}
\end{equation}
where $I_{i,j}$ is a path from $P_i$ to $P_j$ in the set
$R(C_z)=\{P_i\mid i=0,1,1/z,\infty\}$ of ramification points of $\sigma$ 
with
$$-\frac{\pi}{2} <\arg(v),\arg(1-v),\arg(1-vz)\le \frac{3\pi}{2},$$
see also Figure \ref{fig:homology-basis}.  
Then they satisfy (\ref{eq:sp-basis-sigma-action}). 
\begin{figure}[htb]
\includegraphics[width=12cm]{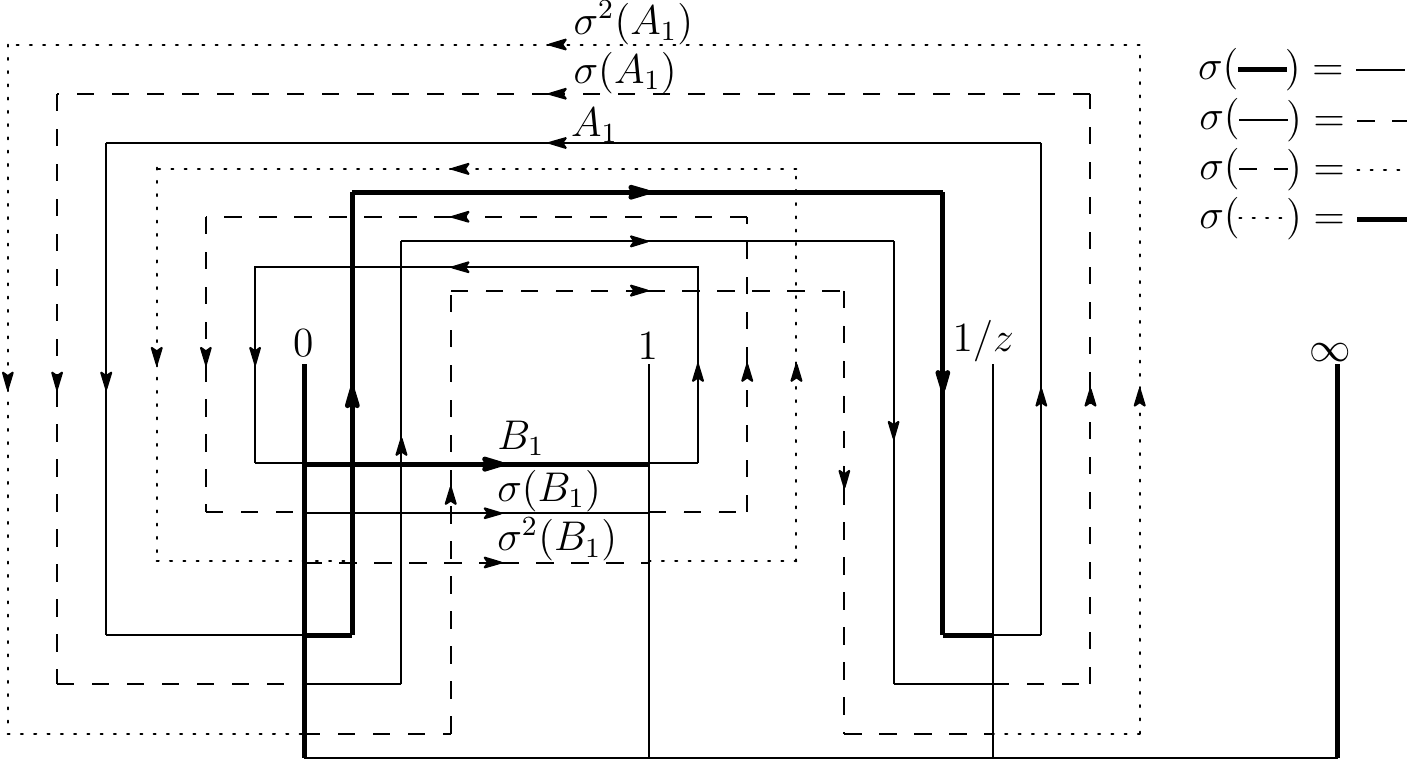}
\caption{A basis of $H_1(C_z,\Z)$}
\label{fig:homology-basis}  
\end{figure}
\end{proposition}

\begin{proposition}
\label{prop:eigenspaces-homology}
Let $H_1^{\sigma^2}(C_z,\Z)$ and $H_1^{-\sigma^2}(C_z,\Z)$ be the eigenspaces 
of the involution $\sigma^2:H_1(C_z,\Z)\to H_1(C_z,\Z)$ 
of eigenvalues $1$ and $-1$, respectively. 
Then these spaces are orthogonal to each other with respect to 
the intersection form. 
The space $H_1^{\sigma^2}(C_z,\Z)$ is generated by 
$$A_1^+= A_1+A_2, \quad B_1^+=B_1+B_2,$$ 
and  $H_1^{-\sigma^2}(C_z,\Z)$  by 
$$A^-_1= A_1-A_2,\quad A^-_2=A_3,\quad 
B_1^-=B_1-B_2,\quad B^-_2=B_3.$$
\end{proposition}
\begin{proof}
Let $A^+$ and $A^-$ be any elements of $H_1^{\sigma^2}(C_z,\Z)$ and 
$H_1^{-\sigma^2}(C_z,\Z)$, respectively. Since the intersection form is 
preserved by any automorphism of $C_z$, we have 
$$A^+\cdot A^-=\sigma^2(A^+)\cdot \sigma^2(A^-)=A^+\cdot (-A^-)=
-A^+\cdot A^-,$$
which means $A^+\cdot A^-=0$. By Proposition \ref{prop:symplectic-basis}, 
it is easy to see that 
$$A^+_1, B^+_1\in H_1^{\sigma^2}(C_z,\Z),\quad 
A^-_1, A^-_2, B^-_1, B^-_2\in H_1^{-\sigma^2}(C_z,\Z).$$
We show that any element $A^-\in H_1^{-\sigma^2}(C_z,\Z)$ 
can be expressed as a linear combination 
of $A^-_1,A^-_2,B^-_1,B^-_2$ over $\Z$.
Since $H_1^{-\sigma^2}(C_z,\Z)\subset H_1(C_z,\Z)$, $A^-$ can be expressed as
$$n_1A_1+n_2A_2+n_3A_3+n_4B_1+n_5B_2+n_6B_3\quad (n_1,\dots,n_6\in \Z).$$
By the orthogonality of the spaces $H_1^{\sigma^2}(C_z,\Z)$ and 
$H_1^{-\sigma^2}(C_z,\Z)$, we have $(A_1+A_2)\cdot A^-=(B_1+B_2)\cdot A^-=0$, 
which yield 
$$-n_4-n_5=n_1+n_2=0.$$
Thus $A^-$ admits the expression 
$$n_1A_1-n_1A_2+n_3A_3+n_4B_1-n_4B_2+n_6B_3=
n_1A^-_1+n_3A^-_2+n_4B^-_1+n_6B^-_2
$$
with coefficients $n_1,n_3,n_4,n_6\in \Z$. Similarly, we can show 
that $H_1^{\sigma^2}(C_z,\Z)$ is generated by $A^+_1=A_1+A_2$, $B^+_1=B_1+B_2$.
\end{proof}

Note that the intersection matrix of the basis 
$A^-_1,A^-_2,B^-_1,B^-_2$ of $H_1^{-\sigma^2}(C_z,\Z)$ is 
\begin{equation}
\label{eq:Prym-Lattice}
\begin{pmatrix} O & -E_{2,1}\\
 E_{2,1} & O
\end{pmatrix},\quad 
O=\begin{pmatrix} 0 & 0 \\ 0 & 0 \end{pmatrix},\ 
E_{2,1}=\begin{pmatrix} 2 & 0 \\ 0 & 1 \end{pmatrix}.
\end{equation}
We introduce a principal subgroup $\L$ of $H_1^{-\sigma^2}(C_z,\Z)$ of index $2$.
\begin{proposition}
\label{prop:Lambda}
Let $\L$ be a subgroup of $H_1(C_z,\Z)$ generated by 
four cycles 
\begin{equation}
\label{eq:basis-Lambda}
\a'_1=A_1-\sigma^2(A_1),\quad \a'_2=\sigma(\a'_1),\quad 
\b_1=B_1+\sigma(B_1),\quad \b_2=\sigma(\b_1).
\end{equation}
Then the intersection matrix of $\a'_1$,$\a'_2$,$\b_1$,$\b_2$ is 
$$
2J_4=2\begin{pmatrix} O & -E_2\\
 E_2 & O
\end{pmatrix},\quad E_2=\begin{pmatrix}
 1 & 0 \\ 0 & 1  
\end{pmatrix}.
$$
We have 
$$(1-\sigma^2)\cdot H_1(C_z,\Z) \subset \L
,
 \quad [H_1^{-\sigma^2}(C_z,\Z):\L]=[\L:(1-\sigma^2)\cdot H_1(C_z,\Z)]=2.
$$
\begin{figure}[htb]
\includegraphics[width=12cm]{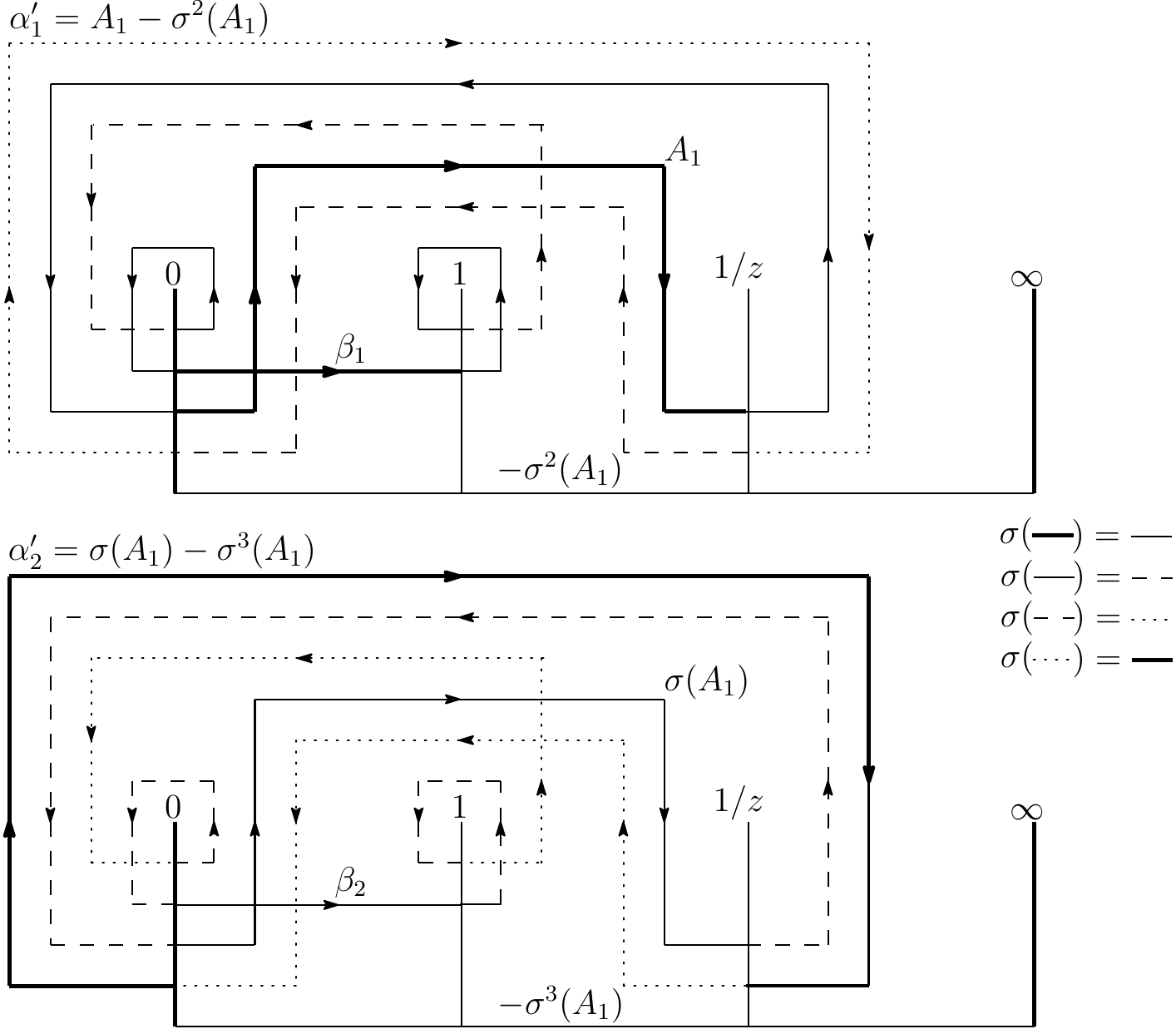}
\caption{Generators of $\L$}
\label{fig:principal-lattice}  
\end{figure}
\end{proposition}
\begin{proof}
Note that 
\begin{align*}
\a'_1&=A^-_1=(1-\sigma^2)(1-\sigma)(I_{0,1}+I_{1,1/z}),\\
\a'_2&=-A^-_1+2A^-_2=\sigma (1-\sigma^2)(1-\sigma)(I_{0,1}+I_{1,1/z}),\\
\b_1&=B^-_1+B^-_2=(1-\sigma^2)(I_{0,1}),\\
\b_2&=B^-_2=\sigma(1-\sigma^2)(I_{0,1}).
\end{align*}
Since they are linear combinations of $A^-_1,A^-_2,B^-_1,B^-_2\in 
H_1^{-\sigma^2}(C_z,\Z)$ of over $\Z$, we have $\L\subset H_1^{-\sigma^2}(C_z,\Z)$.
We show that 
the intersection matrix of $\a'_1$,$\a'_2$,$\b_1$,$\b_2$ becomes $2J_4$.   
By these expressions, it is clear that 
$$\a'_i\cdot \a'_j=\b_i\cdot \b_j=0\quad(1\le i,j\le 2).$$
By (\ref{eq:Prym-Lattice}), we have
\begin{align*}
\b_1\cdot \a'_1&=(B^-_1+B^-_2)\cdot A^-_1=2,\\
\b_1\cdot \a'_2&=(B^-_1+B^-_2)\cdot (-A^-_1+2A^-_2)=0,\\
\b_2\cdot \a'_1&=B^-_2\cdot A^-_1=0,\\
\b_2\cdot \a'_2&=B^-_2\cdot (-A^-_1+2A^-_2)=2.
\end{align*}
We also see these intersection numbers by 
Figure \ref{fig:principal-lattice}. 
By comparing the intersection matrix of $\a'_1$,$\a'_2$,$\b_1$,$\b_2$ with
that of the basis of $H_1^{-\sigma^2}(C_z,\Z)$,
we have $[H_1^{-\sigma^2}(C_z,\Z):\L]=2$.
Since $H_1(C_z,\Z)$ is generated by $\sigma^j\cdot A_1$, 
$\sigma^j\cdot B_1$ $(j=1,2,3)$, 
$(1-\sigma^2)H_1(C_z,\Z)$ is generated by 
$$(1-\sigma^2)\cdot A_1=\a'_1,\ (1-\sigma^2)\sigma \cdot A_1=\a'_2,\ 
(1-\sigma^2)\cdot B_1=\b_1-\b_2,\ (1-\sigma^2)\sigma \cdot B_1
=\b_1+\b_2.$$
Thus  $(1-\sigma^2)H_1(C_z,\Z)$ is a subgroup of $\L$ of index $2$.
\end{proof}

We change the basis $\a'_1,\a'_2,\b_1,\b_2$ of $\L$ into 
\begin{equation}
\label{eq:new-basis}
\a_1=\a'_1-\b_1,\quad \a_2=\a'_2-\b_2,\quad \b_1,\quad \b_2.
\end{equation}
It is easy to check that this new basis also satisfies 
\begin{equation}
\label{eq:sigma-action}
\begin{matrix}
\sigma(\a_1)=\a_2,\quad 
\sigma(\a_2)=-\a_1,\quad
\sigma(\b_1)=\b_2,\quad 
\sigma(\b_2)=-\b_1,\\[3mm]
\a_i\cdot \a_j=\b_i\cdot \b_j=0,\quad 
\a_i\cdot \b_j=-2\d_{i,j}.
\end{matrix}
\end{equation}

\begin{lemma}
\label{lem:paths}
We have 
$$\frac{1}{2}(p_1\a_1+p_2\a_2+q_1\b_1+q_2\b_2)\in H_1^{-\sigma^2}(C_z,\Z)$$
for $p_1,p_2,q_1,q_2\in \Z$ 
if and only if 
$$p_1\equiv p_2\equiv q_1\equiv q_2\bmod 2.$$
We express $(1-\sigma^2)\cdot I_{i,j}\in H_1^{-\sigma^2}(C_z,\Z)$ 
$(i,j\in \{0,1,1/z,\infty\})$ by linear combinations of $\a_1,a_2,\b_1,\b_2$ 
over $\frac{1}{2}\Z$ as 
$$
\begin{matrix}
(1-\sigma^2)\cdot I_{0,1}=\b_1, &(1-\sigma^2)\cdot I_{1,1/z}=\dfrac{1}{2}
(\a_1+\a_2-\b_1+\b_2), \\
(1-\sigma^2)\cdot I_{1/z,\infty}=-\b_2, 
&(1-\sigma^2)\cdot I_{-\infty,0}=\dfrac{1}{2}
(-\a_1-\a_2-\b_1+\b_2).
\end{matrix}
$$
\end{lemma}
\begin{proof}
Note that $A^-_1,A^-_2,B^-_1,B^-_2\in H_1^{-\sigma^2}(C_z,\Z)$ satisfy 
$A^-_1,B^-_1,B^-_2\in \L$, $A^-_2\notin \L$ and  
$$A^-_2=\frac{1}{2}(\a'_1+\a'_2)=
\frac{1}{2}(\a_1+\a_2+\b_1+\b_2)
.$$
Since $[H_1^{-\sigma^2}(C_z,\Z):\L]=2$, 
the first assertion of this lemma holds. 
The cycle $(1-\sigma^2)\cdot I_{0,1}$ is nothing but $\b_1$.
By Figure \ref{fig:principal-lattice}, we have   
$$\a_1'=(1-\sigma^2)\cdot I_{0,1/z}-\sigma(1-\sigma^2)\cdot I_{0,1/z},\quad 
\a_2'=(1-\sigma^2)\cdot I_{0,1/z}+\sigma(1-\sigma^2)\cdot I_{0,1/z}.
$$ 
Thus we have 
\begin{align*}
(1-\sigma^2)\cdot I_{0,1/z}&=\frac{1}{2}(\a'_1+\a_2')
=\frac{1}{2}(\a_1+\a_2+\b_1+\b_2),\\
(1-\sigma^2)\cdot I_{1,1/z}&=(1-\sigma^2)\cdot (I_{0,1/z}- I_{0,1})
=\frac{1}{2}(\a_1+\a_2-\b_1+\b_2).
\end{align*}
Since the cycle $(1-\sigma^2)\cdot I_{1/z,\infty}$ intersects $\a'_2$ with 
the intersection number 
$$\a'_2\cdot ((1-\sigma^2)\cdot I_{1/z,\infty})=2,$$ 
and it does not intersect $\a'_1,\b'_1,\b'_2$, this is equal to 
$-\b'_2=-\b_2$.
Since 
$$(1-\sigma^2)\cdot I_{0,1}+(1-\sigma^2)\cdot I_{1,1/z}
+(1-\sigma^2)\cdot I_{1/z,\infty}+(1-\sigma^2)\cdot I_{-\infty,0}=0$$ 
as elements of $H_0^{-\sigma^2}(C_z,\Z)$, the expression 
$(1-\sigma^2)\cdot I_{-\infty,0}$ is obtained.
\end{proof}

\begin{lemma}
\label{lem:Lambda-index2}
The linear combination 
$p_1\a_1+p_2\a_2+q_1\b_1+q_2\b_2$ $(p_1,p_2,q_1,q_2\in \Z)$ 
belongs to $(1-\sigma^2)H_1(C_z,\Z)$ if and only if $p_1+p_2+q_1+q_2\in 2\Z$.
\end{lemma}
\begin{proof}
The linear span 
$$\{p_1\a_1+p_2\a_2+q_1\b_1+q_2\b_2\mid 
p_1,p_2,q_1,q_2\in \Z,\ p_1+p_2+q_1+q_2\in 2\Z\}$$
is a subgroup of $\L$ of index 2.
Since the generators of $(1-\sigma^2)H_1(C_z,\Z)$ in Proof of 
Proposition \ref{prop:Lambda} are expressed as
$$\a'_1=\a_1+\b_1,\quad \a'_2=\a_2+\b_2,\quad \b_1-\b_2,\quad \b_1+\b_2,
$$
$(1-\sigma^2)H_1(C_z,\Z)$ is included in this linear span. 
Thus they coincide by the property 
$[\L:(1-\sigma^2)H_1(C_z,\Z)]=2$ shown in Proposition \ref{prop:Lambda}. 
\end{proof}
\section{The Abel-Jacobi $\L$-map}
\label{sec:abel-Jacobi}
For a compact Riemann surface $C$, its Jacobi variety $J(C)$ is defined by 
the quotient space $H^0(C,\W)^*/H_1(C,\Z)$, 
where $H^0(C,\W)^*$ is regarded as a subspace of $H_1(C,\C)$ by 
the duality between $H^1(C,\C)(\simeq H^0(C,\W)\oplus \overline{H^0(C,\W)})$ 
and $H_1(C,\C)$ via path integrals of $1$-forms along cycles.
The Abel-Jacobi map is defined by 
$$\jmath:C\ni P \mapsto 
\Big(H^0(C,\W)\ni \f \mapsto \int_{P_0} ^P \f \in \C\Big)\in J(C),
$$
where $P_0$ is a fixed initial point of $C$. 
Here note that the integral $\int_{P_0} ^P \f$ depends on paths 
connecting $P_0$ and $P$, however $\int_{P_0} ^P \f$ is uniquely determined as 
an element of the Jacobi variety $J(C)$. 
When an involution acts on $C$, there are detailed studies 
of the Prym variety 
$H^0_{-}(C,\W)^*/H_1^{-}(C,\Z)$ and the map 
$$C\ni P \mapsto 
\Big(H^0_{-}(C,\W)\ni \f \mapsto \int_{P_0} ^P 2\f \in \C\Big)\in 
H^0_{-}(C,\W)^*/H_1^{-}(C,\Z),
$$
in \cite{F}, 
where 
$H^0_{-}(C,\W)^*$ and $H_1^{-}(C,\Z)$ are the $(-1)$-eigenspaces of 
$H^0(C,\W)^*$ and $H_1(C,\Z)$ under the actions induced from the involution, 
respectively. 

In this section, we study the structures of the quotient space 
\begin{equation}
\label{eq:Lambda-quot}
J_\L(C)=H^0_{-\sigma^2}(C_z,\W)^*/\Lambda
\end{equation}
and a map 
\begin{equation}
\label{eq:AJ-Lambda}
\jmath_\L:C_z\ni P\mapsto 
\Big(H^0_{-\sigma^2}(C_z,\W)\ni \f \mapsto (1-\sigma^2)
\int_{P_0} ^P \f \in \C\Big)\in J_\L(C),
\end{equation}
where $P_0$ is the point of $C_z$ with $(v,w)=(0,0)$. 
We call $\jmath_\L$ the Able-Jacobi $\L$-map. 

At first, we represent $J_\L(C)$ as a quotient space of 
$\C^2$ modulo a lattice in it. 
For elements $\a_1,\a_2\in \L$ given in  Proposition \ref{prop:Lambda},  
let $\int_{\a_i}$ $(i=1,2)$ be elements of the dual space of 
$H^0_{-\sigma^2}(C_z,\W)\oplus \overline{H^0_{-\sigma^2}(C_z,\W)}
\simeq H^1_{-\sigma}(C_z,\C)$ 
given by 
$$\int_{\a_i}:
H^0_{-\sigma^2}(C_z,\W)\oplus \overline{H^0_{-\sigma^2}(C_z,\W)}
\ni \f \mapsto 
\int_{\a_i}\f\in \C.
$$
Similarly, we have $\int_{\b_i}$ $(i=1,2)$ for $\b_1,\b_2\in \L$. 
We can identify $\int_{\a_i}$, $\int_{\b_i}$ $(i=1,2)$
with $\a_i,\b_i\in H_1^{-\sigma^2}(C_z,\C)$.

\begin{lemma}
\label{lem:lattices}
Not only $\int_{\a_i}$ $(i=1,2)$  but also $\int_{\b_i}$ $(i=1,2)$
generate $H^0_{-\sigma^2}(C_z,\W)^*$. 
\end{lemma}
\begin{proof}
Suppose that $H^0_{-\sigma^2}(C_z,\W)^*$ is not generated by 
$\int_{\a_i}$ $(i=1,2)$. 
Then there exists a non-zero $\f\in H^0_{-\sigma^2}(C_z,\W)$ such that 
$$\int_{\a_1} \f=\int_{\a_2} \f=0.$$
By $\f\in 
H^1_{-\sigma^2}(C_z,\C)
\simeq H_1^{-\sigma^2}(C_z,\C)^*$ and the isomorphism 
$H_1^{-\sigma^2}(C_z,\C)^*\simeq H_1^{-\sigma^2}(C_z,\C)$ with respect to 
the intersection form,
$\f$ is identified with $\a_\f\in H_1^{-\sigma^2}(C_z,\C)$ satisfying 
$$A^{-}\cdot \a_\f=\int_{A^{-}} \f$$
for any $A^{-}\in H_1^{-\sigma^2}(C_z,\C)$.
Since 
$$\a_1\cdot \a_\f=\int_{\a_1} \f=0, \quad 
\a_2\cdot \a_\f=\int_{\a_2} \f=0,
$$
$\a_\f$ is expressed as a linear combination $c_1\a_1+c_2\a_2$ $(c_1,c_2\in \C)$.
The complex conjugate $\overline{\f}$ of $\f$ is identified with 
$\a_{\overline{\f}}
={c'_1}\a_1+{c'_2}\a_2\in H_1^{-\sigma^2}(C_z,\C)$ 
$(c'_1,c'_2\in \C)$
by properties
$$\int_{\a_1} \overline{\f}=\int_{\a_2}\overline{\f}=0.
$$
By the compatibility of the intersection forms 
on homology and cohomology groups, we have 
$$\iint_{C_z} \f\wedge \bar \f =\a_\f\cdot \a_{\overline{\f}}
=(c_1\a_1+c_2\a_2)\cdot ({c'_1}\a_1+{c'_2}\a_2)=0. 
$$
However, this contradicts the positivity of the hermitian form 
$$\i \iint_{C_z} \f\wedge \overline{\f}$$
on the space $H^0_{-\sigma^2}(C_z,\W)$. We can similarly show the statement 
for $\int_{\b_1},\int_{\b_2}$.
\end{proof}

By using a basis $\tr(\int_{\b_1}, \int_{\b_2})$ of 
$H^0_{-\sigma^2}(C_z,\W)^*$, we embed $\L$ into $\C^2$, and 
represent $J_\L(C_z)$ as the quotient of $\C^2$ modulo 
the image of $\L$ under this embedding. 
Since $\b_i$ is identified with $\int_{\b_i}$, $\b_1$ and $\b_2$ are 
represented by the unit row vectors $e_1=(1,0)$ and $e_2=(0,1)$, 
respectively.
By Lemma \ref{lem:lattices}, 
$\int_{\a_i}$ can be expressed as 
a linear combination of $\int_{\b_1}, \int_{\b_2}$. 
We set 
$$
\begin{pmatrix} \int_{\a_1}\\ \int_{\a_2}
\end{pmatrix}
=\begin{pmatrix}
\tau_{1,1} & \tau_{1,2}\\
\tau_{2,1} & \tau_{2,2}
\end{pmatrix}
\begin{pmatrix} \int_{\b_1}\\ \int_{\b_2}
\end{pmatrix}.
$$
Let $(\f_1,\f_2)$ be the basis of $H^0_{-\sigma^2}(C_z,\W)$
dual  to $\tr(\int_{\b_1},\int_{\b_2})$, which satisfies 
\begin{equation}
\label{eq:dual-Lambda}
\sint_{\b_i} (\f_j)=\int_{\b_i} \f_j=\d_{i,j}.
\end{equation}
Then we have 
$$\tau_{i,j}=\sint_{\a_i}(\f_j)=\int_{\a_i}\f_j.$$

\begin{theorem}
\label{th:period-matrix}
We have 
$$
\begin{pmatrix}
\tau_{1,1} & \tau_{1,2}\\
\tau_{2,1} & \tau_{2,2}
\end{pmatrix}=
\begin{pmatrix}
\tau & 0\\
0& \tau
\end{pmatrix}, 
\quad 
\tau
=\frac{\int_{\a_1} \eta_1}{\int_{\b_1} \eta_1}
=\frac{\int_{\a_2} \eta_2}{\int_{\b_2} \eta_2}
\in \H.
$$
In particular, $J_\L(C_z)$ decomposes into the direct sum 
$$\la \sint_{\b_1}\ra_\C/\la \a_1,\b_1\ra_\Z\oplus 
\la \sint_{\b_2}\ra_\C/\la \a_2,\b_2\ra_\Z, 
$$
and each component is isomorphic to a complex torus 
$T_\tau=\C/(\Z\tau+\Z),$
where $\la \sint_{\b_i}\ra_\C$ denotes the $\C$-span of  
$\sint_{\b_i}$ and $\la \a_i,\b_i\ra_\Z$
denotes the $\Z$-span of $\a_i,\b_i$.
\end{theorem}
\begin{proof}
Regard $\f_1,\f_2$ as elements of $H_1^{-\sigma^2}(C_z,\C)^*$
and $H_1^{-\sigma^2}(C_z,\C)^*$ as the space isomorphic to 
$H_1^{-\sigma^2}(C_z,\C)$ with respect to the intersection form. 
Then we can express $\f_1,\f_2$ 
as linear combinations of $\a_1,\a_2,\b_1,\b_2\in H_1^{-\sigma^2}(C_z,\C)$ over $\C$ as : 
$$
\f_j=z_{1,j}\a_1+z_{2,j}\a_2+z_{3,j}\b_1+z_{4,j}\b_2\quad (j=1,2).
$$
Since
\begin{align*}
\d_{i,j}
=&\int_{\b_i} \f_j=\beta_i\cdot 
(z_{1,j}\a_1+z_{2,j}\a_2+z_{3,j}\b_1+z_{4,j}\b_2)=
2(z_{1,j}\d_{i,1}+z_{2,j}\d_{i,2}),
\\
(\tau_{i,1}\d_{1,j}+\tau_{i,2}\d_{2,j})
=&(\tau_{i,1}\sint_{\b_1}+\tau_{i,2}\sint_{\b_2})(\f_j)
 =\sint_{\a_i}(\f_j)=\int_{\a_i} \f_j
\\
 =&\a_i\cdot 
(z_{1,j}\a_1+z_{2,j}\a_2+z_{3,j}\b_1+z_{4,j}\b_2)=-2(z_{3,j}\d_{i,1}+
z_{4,j}\d_{i,2}),
\end{align*}
the expressions of $\f_j$ reduce to 
\begin{equation}
\label{eq:phi-hom-exp}
\f_1=\frac{1}{2}(\a_1-\tau_{1,1}\b_1-\tau_{2,1}\b_2),\quad 
\f_2=\frac{1}{2}(\a_2-\tau_{1,2}\b_1-\tau_{2,2}\b_2).
\end{equation}
By the compatibility of the intersection forms on homology and cohomology 
groups, we have
$$0=\iint_{C_z} \f_1\wedge \f_2 =\frac{1}{4}
(\a_1-\tau_{1,1}\b_1-\tau_{2,1}\b_2)\cdot 
(\a_2-\tau_{1,2}\b_1-\tau_{2,2}\b_2)=\frac{1}{2}(\tau_{1,2}-\tau_{2,1}),
$$
which means that 
$\tau$ is symmetric.

We consider the pull backs $\sigma^*(\f_1)$, $\sigma^*(\f_2)$ 
under the action $\sigma$, which can be expressed as linear combinations of 
$\f_1,\f_2$. Since 
$$\int_{\b_1} \sigma^*(\f_j)=\int_{\sigma(\b_1)} \f_j=\int_{\b_2} \f_j=\d_{2,j},
\quad 
\int_{\b_2} \sigma^*(\f_j)=\int_{\sigma(\b_2)} \f_j=-\int_{\b_1} \f_j=-\d_{1,j},
$$
we have 
$$\sigma^*(\f_1)=-\f_2,\quad \sigma^*(\f_2)=\f_1.$$
By using (\ref{eq:phi-hom-exp}), we express them as 
\begin{align*}
\sigma^*(\f_1)&=\frac{1}{2}\sigma^{-1}
(\a_1-\tau_{1,1}\b_1-\tau_{2,1}\b_2)
=\frac{-1}{2}(\a_2-\tau_{1,1}\b_2+\tau_{2,1}\b_1),\\
\sigma^*(\f_2)&=\frac{1}{2}\sigma^{-1}(\a_2-\tau_{1,2}\b_1-\tau_{2,2}\b_2)
=\frac{1}{2}(\a_1+\tau_{1,2}\b_2-\tau_{2,2}\b_1).
\end{align*}
Thus we have relations 
$$\tau_{1,1}=\tau_{2,2}=\tau,\quad \tau_{2,1}=-\tau_{1,2},
$$
and 
$\f_j=\a_j-\tau\b_j$.

Since 
$$0<\i\iint_{C_z} \f_1\wedge \overline{\f_1} 
=\frac{\i}{4}\cdot (\a_1-\tau\b_1)\cdot \overline{(\a_1-\tau\b_1)}
=\frac{\tau-\overline{\tau}}{2\i}=\im(\tau),
$$
$\tau$ belongs to $\H$.

Since 
$$\sigma^*(\f_1-\i\f_2)=-\f_2-\i\f_1=-\i\cdot (\f_1-\i\f_2),\quad 
\sigma^*(-\i\f_1+\f_2)=\i\f_2+\f_1=\i\cdot (-\i\f_1+\f_2),
$$ 
$\f_1-\i\f_2$ and $-\i\f_1+\f_2$ are eigenvectors of $\sigma$ 
of eigenvalues $-\i$ and $\i$, and coincide with 
non-zero constant multiples of $\eta_1$ and $\eta_2$, respectively.
By $$\int_{\a_1}\f_2=\int_{\b_1}\f_2=\int_{\a_2}\f_1=\int_{\b_2}\f_1=0,$$
we have
\begin{align*}
\tau&=\int_{\a_1}\f_1
=\frac{\int_{\a_1}\f_1}{\int_{\b_1}\f_1}
=\frac{\int_{\a_1}\f_1-\i\f_2}{\int_{\b_1}\f_1-\i\f_2}
=\frac{\int_{\a_1}\eta_1}{\int_{\b_1}\eta_1},\\
&=\int_{\a_2}\f_2
=\frac{\int_{\a_2}\f_2}{\int_{\b_2}\f_2}
=\frac{\int_{\a_2}-\i\f_1+\f_2}{\int_{\b_2}-\i\f_1+\f_2}
=\frac{\int_{\a_2}\eta_2}{\int_{\b_2}\eta_2},
\end{align*}
which complete Proof of Theorem \ref{th:period-matrix}.
\end{proof}

\begin{cor}
  \label{cor:ex-dual-basis}
We express $\f_1,\f_2$ forming the dual basis of $H^1_{-\sigma^2}(C_z,\W)$ to 
$\tr(\int_{\b_1},\int_{\b_2})$ as linear combinations
\begin{equation}
\label{eq:ex-dual-basis}
\f_1=\frac{1}{2\int_{\b_1}\eta_1}\eta_1+ \frac{1}{2\int_{\b_1}\eta_2}\eta_2,
\quad   
\f_2=\frac{1}{2\int_{\b_2}\eta_1}\eta_1+\frac{1}{2\int_{\b_2}\eta_2}\eta_2.
\end{equation}
of $\eta_1$ and $\eta_2$ satisfying 
$\sigma^*(\eta_1)=-\i\eta_1$ and $\sigma^*(\eta_2)=\i\eta_2$.
\end{cor}
\begin{proof}
We have  shown in Proof of \ref{th:period-matrix} that  
$\f_1-\i\f_2$, $-\i\f_1+\f_2$ are eigenvectors of $\sigma$ of eigenvalues
$-\i$ and $\i$.
Thus $\f_1$ and $\f_2$ admit expressions 
\begin{align*}
  \f_1&=
\frac{\f_1-\i\f_2}{2\int_{\beta_1}(\f_1-\i\f_2)}
+\frac{\i(-\i\f_1+\f_2)}{2\int_{\beta_2}(-\i\f_1+\f_2)}=
\frac{\eta_1}{2\int_{\beta_1}\eta_1}
 +\frac{\i\eta_2}{2\int_{\sigma(\beta_1)}\eta_2}
=\frac{\eta_1}{2\int_{\b_1}\eta_1}+ \frac{\eta_2}{2\int_{\b_1}\eta_2}        ,\\
  \f_2&=
\frac{\i(\f_1-\i\f_2)}{2\int_{\beta_1}(\f_1-\i\f_2)}
+\frac{-\i\f_1+\f_2}{2\int_{\beta_2}(-\i\f_1+\f_2)}                
=\frac{-\i\eta_1}{2\int_{\sigma(\beta_2)}\eta_1}
        +\frac{\eta_2}{2\int_{\beta_2}\eta_2}
 =\frac{\eta_1}{2\int_{\b_2}\eta_1}+ \frac{\eta_2}{2\int_{\b_2}\eta_2},        
\end{align*}
by $\int_{\b_i}\f_j=\d_{i,j}$ and $\int_{\sigma(\b_i)}\eta_j=
\int_{\b_i}\sigma^*(\eta_j)$.
\end{proof}

Next, we consider the Abel-Jacobi $\L$-map in \eqref{eq:AJ-Lambda}. 
Note that the integral 
$$(1-\sigma^2)\int_{P_0}^P \f=\int_{\overrightarrow{P_0P}} \f
-\int_{\sigma^2\cdot \overrightarrow{P_0P}} \f,$$ 
depends on the choice of a path $\overrightarrow{P_0P}$ from $P_0$ to $P$, 
where $\sigma^2\cdot \overrightarrow{P_0P}$ is the 
$\sigma^2$-image of $\overrightarrow{P_0P}$.
However, the differences can be expressed as 
$\int_A \f$ for 
$$A\in (1-\sigma^2)\cdot H_1(C_z,\Z)\subset \L,$$ 
the map $\jmath_\L$ is well-defined. Since $\f\in  H^0_{-\sigma^2}(C_z,\W)$, we 
have 
$$(1-\sigma^2)\cdot \int_{P_0}^P \f
=\int_{P_0}^P \f-\sigma^2\cdot \int_{P_0}^P \f
=\int_{P_0}^P \f-\int_{P_0}^P (\sigma^2)^*(\f)
=\int_{P_0}^P 2\f.
$$
By using the basis $(\f_1,\f_2)$ of $H^0_{-\sigma^2}(C_z,\W)$ satisfying
\begin{equation*}
\int_{\a_i}\f_j=\d_{i,j}\tau,\quad 
\int_{\b_i}\f_j=\d_{i,j},
\end{equation*}
for $\a_1,\a_2,\b_1,\b_2\in \L$,  
we regard the map $\jmath_\L$ as 
$$\jmath_\L:C_z\ni P \mapsto (\int_{P_0}^P 2\f_1,\int_{P_0}^P 2\f_2)
\in T_\tau\oplus T_\tau.$$
By compositing $\jmath_\L$ and the projection $T_\tau\oplus T_\tau\to T_\tau$,  
we have maps 
\begin{equation}
\label{eq:SAJ-map}
\jmath_{\L_i}:C_z\ni P \mapsto 
\int_{P_0}^P 2\f_i
\in T_\tau \quad (i=1,2).
\end{equation}
Since 
\begin{equation}
\label{eq:rel-jmaths}
\jmath_{\L_2}(P)=\int_{P_0}^P 2\f_2=
\int_{P_0}^P -2(\sigma)^*(\f_1)=
-\sigma \cdot \int_{P_0}^P 2\f_1=-\jmath_{\L_1}(\sigma\cdot P),
\end{equation}
we can regard $\jmath_{\L_2}$ as the $(-\sigma)$-action of $\jmath_{\L_1}$.
\begin{proposition}
\label{prop:images} 
Let $\widetilde{C_z}$ be the universal covering of $C_z$, and  
let $$\widetilde{\jmath_{\L}}:\widetilde{C_z}\to \C^2$$
be the lift of the map $\jmath_{\L}:C_z\to T_\tau\oplus T_\tau$. 
Then the images of $P_0$, $P_1$, $P_{1/z}$, $P_\infty$ under the map 
$\widetilde{\jmath_\L}$
are 
$$(0,0),\quad (1,0),\quad \big(\frac{\tau+1}{2},\frac{\tau+1}{2}\big),\quad 
\big(\frac{\tau+1}{2},\frac{\tau-1}{2}\big),$$
where the paths from $P_0$ to $P_1$, $P_{1/z}$ and $P_{\infty}$ 
are $I_{0,1}$, $I_{0,1}+I_{1,1/z}$, $I_{0,1}+I_{1,1/z}+I_{1/z,\infty}$, respectively.
\end{proposition}

\begin{proof}
It is clear that $\jmath_{\L}(P_0)=(0,0)$. 
By Lemma \ref{lem:paths}, we have 
$$\big((1-\sigma^2)\cdot \int_{P_0}^{P_1}\f_1,
(1-\sigma^2)\cdot \int_{P_0}^{P_1}\f_2\big)
=\big(\int_{\b_1}\f_1,\int_{\b_1}\f_2\big)=(1,0),$$ 
which is equal to the zero of 
$T_\tau$. We have also 
\begin{align*}
&\big((1-\sigma^2)\cdot\int_{P_0}^{P_{1/z}}\f_1, 
(1-\sigma^2)\cdot\int_{P_0}^{P_{1/z}}\f_2\big) 
=\big(\frac{\tau+1}{2},\frac{\tau+1}{2}\big),\\ 
&\big((1-\sigma^2)\cdot\int_{P_0}^{P_{\infty}}\f_1, 
(1-\sigma^2)\cdot\int_{P_0}^{P_{\infty}}\f_2\big) 
=\big(\frac{\tau+1}{2},\frac{\tau-1}{2}\big), 
\end{align*}
since $(1-\sigma^2)\cdot (I_{0,1}+I_{1,1/z})$ and 
$(1-\sigma^2)\cdot (I_{0,1}+I_{1,1/z}+I_{1/z,\infty})$ are 
equal to $\dfrac{\a_1+\a_2+\b_1+\b_2}{2}$ and 
$\dfrac{\a_1+\a_2+\b_1-\b_2}{2}$ by Lemma \ref{lem:paths}.
\end{proof}

\begin{proposition}
\label{prop:inje-jmath-L}
The map $\jmath_{\L_i}:C_z \to T_\tau$ is of degree $2$, and  
the map $\jmath_\L:C_z\to J_\L(C_z)$ is generically injective.
\end{proposition}
\begin{proof}

By Proposition \ref{prop:images}, the points $P_0,P_1\in C_z$ 
belong to the preimage of $(0,0)\in T_\tau\oplus T_\tau$ 
under of the map $\jmath_{\L}$. 
Any different point $P$ from them is never mapped $(0,0)$ since 
the $(1-\sigma^2)$-image of any path from $P_0$ to $P$ does not 
belong to $\L$. Thus we have 
$$\jmath_{\L}^{-1}(0,0)=\{P_0,P_1\}.$$
By regarding $\jmath_{\L_2}$  as the $(-\sigma)$-action of $\jmath_{\L_1}$ 
as in (\ref{eq:rel-jmaths}), 
we see that 
$$\jmath_{\L_1}(P)=0 \Leftrightarrow \jmath_{\L_2}(P)=0.$$ 
Thus the preimage of $0\in T_\tau$ under $\jmath_{\L_i}$ ($i=1,2$) 
is $\{P_0,P_1\}$. We show that these points are different from 
ramification points of $\jmath_{\L_i}$. Here note that $P\in C_z$ is 
a ramification point of $\jmath_{\L_i}$ if and only if $\f_i(P)=0$. 
By substituting $P_0$ and $P_1$ in 
the expressions (\ref{eq:ex-dual-basis}) of $\f_i$,
we have 
$$\f_i(P_0)=
\frac{\eta_1(P_0)}{\int_{\b_i} \eta_1}\ne 0,\quad 
\f_i(P_1)=
\frac{\eta_2(P_1)}{\int_{\b_i} \eta_2}\ne 0,
$$
by the facts (\ref{eq:zero-div}) 
about the zero divisors of $\eta_1$ and $\eta_2$ given 
in Proposition \ref{prop:sigma-action}.
Hence we conclude that $\jmath_{\L_i}$ is of degree $2$. 

By comparing the genus of $C_z$ with that of $T_\tau$ by Hurwitz's formula, 
we see that there exist four ramification points of index $2$ for 
$\jmath_{\L_i}$. 
We show that there is no intersection 
between the set of zeros of $\f_1$ and that of $\f_2$. 
Suppose that $P$ is in this intersection. 
By substituting $P$ in (\ref{eq:ex-dual-basis}),  
we have 
$$\f_1(P)=
\frac{\eta_1(P)}{\int_{\b_1} \eta_1}+\frac{\eta_2(P)}{\int_{\b_1} \eta_2}
=0, \quad 
\f_2(P)=
\frac{\eta_1(P)}{\int_{\b_2} \eta_1}+\frac{\eta_2(P)}{\int_{\b_2} \eta_2}
=\i\frac{\eta_1(P)}{\int_{\b_1} \eta_1}-\i\frac{\eta_2(P)}{\int_{\b_1} \eta_2}
=0.
$$
These mean that $\eta_1(P)=\eta_2(P)=0$,  
which contradict (\ref{eq:zero-div}) 
in Proposition \ref{prop:sigma-action}.

By making the resolution of $\jmath_\L(C_z)$ if necessary, 
we regard $\jmath_{\L}$ as a map from $C_z$ to a compact Riemann surface
$\widehat{\jmath_\L(C_z)}$. 
Since there is no common zero of $\f_1$ and $\f_2$, 
this map is a covering without ramification points. 
By Hurwitz's formula, we have two possibilities: 
\begin{itemize}
\item[(1)] $\widehat{\jmath_\L(C_z)}$ is of genus $2$ and 
$\jmath_\L$ is of degree $2$,
\item[(2)] $\widehat{\jmath_\L(C_z)}$ is of genus $3$ and 
$\jmath_\L$ is of degree $1$.
\end{itemize}
In case of (1), 
the projection from $\widehat{\jmath_\L(C_z)}$ to $T_\tau$ 
is of degree $1$ since $\jmath_{\L_i}$ is of degree $2$. 
However, it contradicts the fact that $T_\tau$ is of genus $1$.  
Hence the case (2) happens.
Note that not only $\jmath_{\L}(P_0)$ and $\jmath_{\L}(P_1)$ but 
also $\jmath_{\L}(P_{1/z})$ and $\jmath_{\L}(P_\infty)$ 
are separated in $\widehat{\jmath_\L(C_z)}$.
\end{proof}

\section{Relations between $[f_1(x),\dots,f_4(x)]\in \P^3$ and 
$(\jmath_{\L}(P),\tau)\in J_{\L}(C_z)\times \H$}
\label{sec:relation}
Suppose that $(x_1,x_2)\in \dot U\cap \R^2$ and that
$$z=\frac{1-x_1-x_2}{(1-x_1)(1-x_2)}$$
is fixed. Note that 
\begin{equation}
\label{eq:assum-x1-x2}
0< x_1<1,\ 0< x_2<1,\ 0< 1-x_1-x_2<1,\ 0<z<1.
\end{equation}
Let $P_{1-x_i}$ $(i=1,2)$ be the point of 
$C_z:w^4=v^3(1-v)(1-vz)$ with $v=1-x_i$ and positive real $w$-coordinate. 
We relate 
$\tau=\int_{\b_1}\f_1 \in \H$, $\jmath_{\L}(P_{1-x_1})$,
$\jmath_{\L}(P_{1-x_2})$ 
to the integrals $f_1(x),\dots,f_4(x)$ in 
(\ref{eq:periods}).

\begin{theorem}
\label{th:rel-Prym-Periods}
We have 
\begin{align}
\label{eq:tau-express}
\tau&=-\frac{f_1(x)}{f_2(x)}-\i,\\
\label{eq:P(1-x1)-express}
\jmath_{\L}(P_{1-x_1})
&=
\big( \frac{1-\i}{4}\cdot \frac{f_3(x)-f_4(x)}{f_2(x)}+\frac{1}{2},
 \frac{1+\i}{4}\cdot \frac{f_3(x)+f_4(x)}{f_2(x)}-\frac{\i}{2}
\big),\\
\label{eq:P(1-x2)-express}
\jmath_{\L}(P_{1-x_2})
&=
\big( \frac{1-\i}{4}\cdot  \frac{f_4(x)-f_3(x)}{f_2(x)}+\frac{1}{2},
 \frac{1+\i}{4}\cdot  \frac{f_3(x)+f_4(x)}{f_2(x)}-\frac{\i}{2}
\big).
\end{align}
\end{theorem}
\begin{proof}
By Lemma \ref{lem:paths},  the ratios $f_1(x)/ f_2(x)$ and 
$f_3(x)/ f_2(x)$ are expressed as
$$\frac{f_1(x)}{f_2(x)}=
\frac{\int_{P_{\infty}}^{P_0} \ex(3/8)\eta_1}{(\i/\sqrt2) 
\int_{P_1}^{P_{1/z}} \ex(-1/4) 
\eta_1}=(-1+\i)\frac{\int_{-\a_1-\a_2-\b_1+\b_2} \eta_1}{2\int_{-\b_2 }\eta_1}
=-\tau-\i,
$$
$$\frac{f_3(x)}{f_2(x)}=
\frac{\ex(3/8)\int_{P_0}^{P_{1-x_1}} \eta_1}{(\i/\sqrt2) 
\int_{P_1/z}^{P_{\infty}}\ex(-1/4) \eta_1} =
\frac{(-1+\i)\int_{P_0}^{P_{1-x_1}} \eta_1}{\int_{-\b_2/2}\eta_1}
=
\frac{2(1+\i)\int_{P_0}^{P_{1-x_1}} \eta_1}{\int_{\b_1}\eta_1}.
$$
Thus we have the expression of $\tau$ by $f_1(x),f_2(x)$,  
and 
$$
\frac{\int_{P_0}^{P_{1-x_1}} \eta_1}{\int_{\b_1}\eta_1}=
\frac{1-\i}{4}\cdot \frac{f_3(x)}{f_2(x)},
\quad 
\frac{\int_{P_0}^{P_{1-x_2}} \eta_1}{\int_{\b_1}\eta_1}=
\frac{1-\i}{4}\cdot \frac{f_4(x)}{f_2(x)}.
$$
By Corollary  \ref{cor:ex-dual-basis}, we have
\begin{align*}
\int_{P_0}^{P_{1-x_1}}2\f_1&=
\frac{\int_{P_0}^{P_{1-x_1}} \eta_1}{\int_{\b_1}\eta_1}+
\frac{\int_{P_0}^{P_{1-x_1}} \eta_2}{\int_{\b_1}\eta_2}
=\frac{\int_{P_0}^{P_{1-x_1}} \eta_1}{\int_{\b_1}\eta_1}+
\frac{\int_{P_1}^{P_{1-x_2}} \eta_1}{\int_{-\b_1}\eta_1}\\
&=\frac{\int_{P_0}^{P_{1-x_1}} \eta_1}{\int_{\b_1}\eta_1}
-\frac{\int_{P_0}^{P_{1-x_2}} \eta_1-\int_{P_0}^{P_1} \eta_1}{\int_{\b_1}\eta_1}
=\frac{1-\i}{4}\frac{f_3(x)}{f_2(x)}
-\big(\frac{1-\i}{4}\frac{f_4(x)}{f_2(x)}-\frac{1}{2}\big),
\\
\int_{P_0}^{P_{1-x_1}}2\f_2&=
\frac{\int_{P_0}^{P_{1-x_1}} \eta_1}{\int_{\b_2}\eta_1}+
\frac{\int_{P_0}^{P_{1-x_1}} \eta_2}{\int_{\b_2}\eta_2}
=
\frac{\int_{P_0}^{P_{1-x_1}} \eta_1}{\int_{\sigma(\b_1)}\eta_1}+
\frac{\int_{P_0}^{P_{1-x_1}} \eta_2}{\int_{\sigma(\b_1)}\eta_2}\\
&=
\frac{\int_{P_0}^{P_{1-x_1}} \eta_1}{-\i\int_{\b_1}\eta_1}
+\frac{\int_{P_0}^{P_{1-x_1}} \eta_2}{\i \int_{\b_1}\eta_2}
=\frac{1+\i}{4}\frac{f_3(x)}{f_2(x)}
+\big(\frac{1+\i}{4}\frac{f_4(x)}{f_2(x)}-\frac{\i}{2}\big),
\end{align*}
where we use the variable change $v_1\mapsto \dfrac{1-v_1}{1-v_1z}$
to convert the integral of $\eta_2$ into that of $\eta_1$. 
Here note that $\eta_2$ is transformed into $\eta_1$ times some factors 
by this variable change as in identities for $\iint_{D_4}u(t,x)dt$ 
in Proposition \ref{prop:a=c1=c2}.
However, these factors are canceled since they appear in numerator and 
denominator.
To obtain (\ref{eq:P(1-x2)-express}), exchange $x_1$ and $x_2$ for 
(\ref{eq:P(1-x1)-express}).
\end{proof}

\begin{cor}
\label{cor:preimages}
Under the assumption \eqref{eq:assum-x1-x2},  
the maps $\jmath_{\L_1}$ and $\jmath_{\L_2}$ satisfy 
\begin{equation}
\label{eq:symmetry}
\jmath_{\L_1}(P_{1-x_1})=\jmath_{\L_1}(\sigma^2(P_{1-x_2})),\quad 
\jmath_{\L_2}(P_{1-x_1})=\jmath_{\L_2}(P_{1-x_2}).
\end{equation}
\end{cor}

\begin{proof}
We have 
$$\jmath_{\L_1}(\sigma^2(P))=-\jmath_{\L_1}(P)$$
for $P\in C_z$ by $\sigma^2(\f_i)=-\f_i$ $(i=1,2)$, and $-1/2$  is 
equal to $1/2$ as elements of $T_\tau$. 
The identities in (\ref{eq:symmetry}) follow by the equalities 
(\ref{eq:P(1-x1)-express}) and (\ref{eq:P(1-x2)-express}). 
\end{proof}

\begin{cor}
\label{cor:ram-pts}
Under the assumption \eqref{eq:assum-x1-x2},  
the ramification points of $\jmath_{\L_1}$ and those of $\jmath_{\L_2}$ are
$P_{v_+}$, $\sigma^2(P_{v_+})$, $\sigma(P_{v_-})$, $\sigma^3(P_{v_-})$ 
and 
$\sigma(P_{v_+})$, $\sigma^3(P_{v_+})$, $P_{v_-}$, $\sigma^2(P_{v_-})$, 
respectively, where 
$$v_\pm=\frac{1\pm\sqrt{1-z}}{z},$$
$P_{v_-}$ is the point of $C_z$ with $v=v_-$ and positive real $w$-coordinate, 
$P_{v_+}$ is the point of $C_z$ with $v=v_+$ and $\arg(w)=-\frac{\pi}{2}$.
\end{cor}

\begin{proof}
We consider ramification points of $\jmath_{\L_2}$. 
By $\jmath_{\L_2}(P_{1-x_1})=\jmath_{\L_2}(P_{1-x_2})$, 
the $v$-coordinates of the ramification points of $\jmath_{\L_2}$ should 
satisfy the equation 
$$zv^2-2v+1=0$$
obtained by the substitution $1-x_1=1-x_2=v$ into 
$$\frac{1-x_1-x_2}{(1-x_1)(1-x_2)}=z.$$
Its solutions are $v_\pm$. Since $v_-$ belongs to the open interval 
$(0,1)$ under the condition \eqref{eq:assum-x1-x2}, 
$P_{v_-}$ is a ramification point of $\jmath_{\L_2}$. 
Note that the $1$-form $\f_2$ vanishes at $P_{v_-}$.  
By acting $\sigma^2$ on \eqref{eq:ex-dual-basis} in Corollary 
\ref{cor:ex-dual-basis},
we see that $\f_2$ vanishes at $\sigma^2(P_{v_-})$.   
Thus $\sigma^2(P_{v_-})$ is a ramification point of $\jmath_{\L_2}$. 
On the other hand, 
$\f_2$ does not vanish at $\sigma(P_{v_-})$ and $\sigma^3(P_{v_-})$ 
by the changes of coefficients under the actions $\sigma$ and $\sigma^3$ 
on \eqref{eq:ex-dual-basis},  
these points are different from a ramification point of $\jmath_{\L_2}$.

Note that $v_+$ belongs to the open interval $(1/z,\infty)$.
We consider whether $\f_2$ vanishes at $P_{v_+}$ by 
tracing $\f_2$ along the lift of a path from $v_-$ to $v_+$ via 
the upper-half space of the $v$-space. Since both of the arguments of $1-v$ and 
$1-vz$ decrease by $\pi$, we have 
$$\arg(w)=-\frac{\pi}{2},\quad
\arg(\frac{1}{w})=\frac{\pi}{2},\quad \arg(\frac{v^2}{w^3})=\frac{3\pi}{2},
$$ 
at $P_{v_+}$. Thus $\f_2$ does not vanish at $P_{v_+}$, and 
this point is not a ramification point of $\jmath_{\L_2}$.
By acting $\sigma$ on $\f_2$ several times, we see that 
$\sigma(P_{v_+})$ and $\sigma^3 (P_{v_+})$ are ramification points of 
$\jmath_{\L_2}$, and that $P_{v_+}$ and $\sigma^2 (P_{v_+})$ are not.
Hence we have the four ramification points of $\jmath_{\L_2}$  
as shown in Proof of Proposition \ref{prop:inje-jmath-L}.

Since $\sigma^*(\f_2)=\f_1$, we obtain 
the assertion for ramification points of $\jmath_{\L_1}$. 
\end{proof}

\begin{cor}
\label{cor:tau-pure-im}
Under the condition \eqref{eq:assum-x1-x2},
 $\tau$ is pure imaginary. 
\end{cor}

\begin{proof}
We can easily show this corollary by \eqref{eq:periods} and 
Theorem \ref{th:rel-Prym-Periods}.  
\end{proof}

\section{The monodromy of Schwarz's map for $\cF_2$}
\label{sec:monod}
We modify Schwarz's map $\mathbf{f}$ for $\cF_2$ by the linear transformation 
\begin{equation}
\label{eq:modified-period-map}
(x_1,x_2)\mapsto \tr(f'_1(x),\dots,f'_4(x))=
Q\tr(f_1(x),\dots f_4(x)),\quad 
Q=\begin{pmatrix}
-1 & -\i & 0 & 0\\
 0 &  1  & 0 & 0\\
 0 &\frac{1}{2} &\frac{1-\i}{4} &\frac{-1+\i}{4}\\
 0 &\frac{-\i}{2} &\frac{1+\i}{4} &\frac{1+\i}{4}\\
\end{pmatrix}.
\end{equation}
Then we see that 
$$f'_1(x)/f'_2(x)\in \H,\quad 
f'_3(x)/f'_2(x),f'_4(x)/f'_2(x)\in T_\tau.
$$

This modification of Schwarz's map transforms the representation matrices 
$M_i^\mu$ with parameters
$\mu=(\i,\i,\i,\i,-1)$
in \cite[Corollary 3.12]{MSTY}  
into $M_i=Q M^\mu_i Q^{-1}$ $(i=1,\dots,5)$, which are 
$$
M_1=
\i\left(\begin{array}{c@{\hspace{3mm}}c@{\hspace{3mm}}c@{\hspace{3mm}}c}
0 & 1 & 0 & 0\\
-1& 0& 0 & 0\\
-1 & 0 & 0& 1\\
0 & 1 & -1 & 0
\end{array}\right),\quad 
M_2=
\i\left(\begin{array}{c@{\hspace{3mm}}c@{\hspace{3mm}}c@{\hspace{3mm}}c}
0 & 1 & 0 & 0\\
-1& 0& 0& 0\\
0 & 0 & 0 & -1\\
0 & 0 & 1 & 0
      \end{array}\right),
    $$
    $$
M_3=
\left(\begin{array}{c@{\hspace{3mm}}c@{\hspace{3mm}}c@{\hspace{3mm}}c}
1 & 2 & 0 & 0\\
0 & 1 & 0 & 0\\
0 & 0 & 1 & 0\\
0 & 0 & 0 & 1\end{array}\right),
\quad
M_4=
\left(\begin{array}{c@{\hspace{3mm}}c@{\hspace{3mm}}c@{\hspace{3mm}}c}
2& 1 & 0 & 0\\
-1& 0 & 0 & 0\\
-1& -1 & 1 & 0\\
0 & 0 & 0 & 1\end{array}\right),
\quad 
M_5=
\left(\begin{array}{c@{\hspace{3mm}}c@{\hspace{3mm}}c@{\hspace{3mm}}c}
2 & 1 & 0& 0\\
-1& 0 & 0 & 0\\
0 & 0 & 1 & 0\\
0 & 0 & 0 & 1\end{array}\right).
$$

\begin{theorem}
\label{th:str-monod}
The group $\cM$ generated by $M_1$,\dots, $M_5$ is 
\begin{equation}
\label{eq:express-monod}
\Big\{\i^{-n_1+n_2} \begin{pmatrix} 
G & O\\
L & J_2^{n_1+n_2}\
\end{pmatrix}
\in \la \i\ra \times SL_4(\Z)\; \Big|\; 
\begin{array}{l} (n_1,n_2)\in (\Z/(2\Z))^2,\ G\in \G_{1,2},\\
l_{11}+l_{12}\equiv l_{21}+l_{22}\equiv n_1 \bmod2\end{array}
 \Big\},
\end{equation}
where $\G_{1,2}$ is the Igusa group given in \S\ref{sec:Schwarz's map}, 
$
L=\begin{pmatrix}
 l_{11} &l_{12} \\
 l_{21} &l_{22} 
 \end{pmatrix}$, and
$J_2=\begin{pmatrix} 0 & -1\\ 1 & 0\end{pmatrix}.$
\end{theorem}
\begin{remark}
\label{rem:(2,2)-block}
The $(2,2)$-block of $\cM$ 
is generated by those of $M_1$ and $M_2$. They satisfy 
$M_1M_2=M_2M_1$, $M_1^2=M_2^2=E_4$, and their product 
$$M_1M_2=\begin{pmatrix}
1 & 0 & 0 & 0 \\
0 & 1 & 0 & 0 \\
0 & 1 &-1 & 0 \\
1 & 0 & 0 &-1 \\
\end{pmatrix}$$
is of order $2$. 
Thus the group structure of 
the $(2,2)$-block of $\cM$ 
is isomorphic to 
$(\Z/(2\Z))^2$, and an isomorphism is given in (\ref{eq:express-monod}).
Note that $(n_1,n_2)\in (\Z/(2\Z))^2$ in (\ref{eq:express-monod}) is  
uniquely determined by the $(2,2)$-block of any element of  $\cM$;
in fact, we give its correspondence as follows.
$$\begin{array}{|c|cccc|}
\hline
\la M_1,M_2\ra & E_4 & M_2 & M_1 & M_1M_2\\
(2,2)\textrm{-block} & E_2 & \i J_2 & -\i J_2& -E_2\\
(n_1,n_2) &(0,0) &(0,1) &(1,0) &(1,1) \\
\hline \end{array}
$$
Note also that $\pm\i J_2$ are of order $2$ though $\pm J_2$ are of order $4$. 
\end{remark}

\begin{proof}
At first, we show that the group $\widetilde \cM$
generated by $M_1$,\dots, $M_5$ and $-E_4$ coincides with
$$
\Big\{\pm\i^{-n_1+n_2} \begin{pmatrix} 
G & O\\
L & J_2^{n_1+n_2}\
\end{pmatrix}
\in \la \i\ra \times SL_4(\Z)\; 
\Big|\; (n_1,n_2)\in (\Z/(2\Z))^2,\ G\in \G_{1,2}, \ \Sigma(L)\in 2\Z\Big\},
$$
where $\Sigma(L)$ denotes the sum of all entries of $L$.
It is easy to check that this set $\widetilde{\cM}'$ admits a group structure
and contains $M_1$,\dots,$M_5$ and $-E_4$. Thus we have $\widetilde{\cM}
\subset \widetilde{\cM}'$. 
We show that any element $M$ of $\widetilde{\cM}'$ can be expressed 
a product of $M_1$,\dots,$M_5$ and $-E_4$. 
Note that the $(2,2)$-block of $\widetilde\cM$ is generated by 
$M_1$ and $M_2$, and it is isomorphic to $(\Z/(2\Z))^2$ 
as seen in  Remark \ref{rem:(2,2)-block}. 
The $(1,1)$-block of $\widetilde\cM$ is generated by those of $M_3$ and $M_5$ 
since 
$\G_{1,2}$ is generated by 
$\begin{pmatrix}
1 & 2\\ 0 & 1
\end{pmatrix}
$ and $J_2$, and 
$M_3M_5=\begin{pmatrix}
J_2^{-1}& O \\
O & E_2  \\
\end{pmatrix}.
$
Hence we can transform any element $M$ in $\widetilde{\cM}'$ 
into 
$$
\begin{pmatrix} 
E_2 & O\\
L & E_2
\end{pmatrix}
$$
by multiplying $M_1,M_2,M_3$ and $M_5$. 
Note that the matrices of this type   
admit an additive group structure with respect to $2\times 2$ matrices $L$'s,
and that this group is normal in  $\widetilde{\cM}'$. 
We give generators of this group in terms of products of 
$M_1,\dots, M_5$ and $-E_4$.
We put 
\begin{align*}
H_1&=M_4M_5^{-1}=
\begin{pmatrix} 
E_2 & O\\
L_1 & E_2
\end{pmatrix},\quad 
L_1=\begin{pmatrix} 
-1 & -1\\
0 & 0
\end{pmatrix},\\ 
H_2&=-(M_3M_5)^2M_1M_2= 
\begin{pmatrix} 
E_2 & O\\
L_2 & E_2
\end{pmatrix},\quad 
L_2=\begin{pmatrix} 
 0 &-1\\
-1 & 0
\end{pmatrix}.
\end{align*}
The $(2,1)$-blocks of $M_2H_1M_2^{-1}$, $(M_3M_5^{-1})H_1(M_3M_5^{-1})^{-1}$
are 
$$
L_3=\begin{pmatrix} 
0 & 0\\
-1 & 1
\end{pmatrix},\quad  
L_4=\begin{pmatrix} 
-1 & 1\\
0 & 0
\end{pmatrix}.
$$ 
The matrices $L_1,\dots,L_4$ span the set of integral $2\times 2$ 
matrices $L$ with $\Sigma(L)\in 2\Z$
since 
$$
(L_1,L_4,L_3,L_2)=
\left(\begin{pmatrix}
1 & 0\\
0 & 0\\
\end{pmatrix},
\begin{pmatrix}
0 & 1\\
0 & 0\\
\end{pmatrix},
\begin{pmatrix}
0 & 0\\
1 & 0\\
\end{pmatrix},
\begin{pmatrix}
0 & 0\\
0 & 1\\
\end{pmatrix}
\right)
\begin{pmatrix}
-1 & -1 & 0 & 0\\
-1 &  1 & 0 &-1\\
 0 &  0 &-1 &-1\\
 0 &  0 & 1 & 0\\
\end{pmatrix}
$$
and the determinant of the last matrix is $-2$. 
Hence we have $\widetilde{\cM}=\widetilde{\cM}'$.

Let $\cM'$ be the set given in (\ref{eq:express-monod}). 
By the case study of $(2,1)$ and $(2,2)$ blocks of elements of $\cM'$, 
we can show that $\cM'$ admits a group structure. 
Since $M_1,\dots,M_5$ belong to $\cM'$, we have $\cM\subset \cM'$.  
It is easy to see that $-E_4\notin \cM'$.  
Thus $-E_4$ does not belong to $\cM$, and $\cM$ is a subgroup of 
$\widetilde{\cM}$ of index $2$. Therefore, we have 
$$
\cM\subset \cM'\subsetneqq \widetilde{\cM}'=\widetilde{\cM}, 
\quad [\widetilde{\cM}:\cM]=2,
$$
which imply $\cM= \cM'$.
\end{proof}

\section{The inverse of Schwarz's map for $\cF_2$}
\label{sec:inverse}

Theta functions with characteristics are defined by 
\begin{equation}
\label{eq:theta-function}
\h_{k,\ell }(y,\tau)=
\sum_{n\in \Z} \exp\big[\pi\i\big((n+\frac{k}{2})^2\tau+2(n+\frac{k}{2})
(y+\frac{\ell }{2})\big)\Big],
\end{equation}
where $(y,\tau)\in \C\times \H$ and $(k,\ell )=(0,0),(0,1),(1,0),(1,1)$.
These functions are holomorphic on $\C\times \H$ and satisfy
basic properties
\begin{align}
\nonumber
\h_{k,\ell }(y+p\tau+q,\tau)&=
(-1)^{kq+\ell p}\exp(-\pi\i(p^2\tau +2py))
\h_{k,\ell }(y,\tau),\\
\nonumber
\h_{k,\ell }(-y,\tau)&=(-1)^{k\ell }\h_{k,\ell }(y,\tau),\\
\label{eq:basic-theta}
\h_{k,\ell }(y+\frac{1}{2},\tau)&=(-1)^{k\ell }\h_{k,1-\ell }(y,\tau),\\
\nonumber
\h_{k,\ell }(y+\frac{\tau}{2},\tau)&=(-\i)^{\ell }\exp(-\pi\i(\frac{\tau}{4}+y))
\h_{1-k,\ell }(y,\tau),\\
\nonumber
\h_{k,\ell }(y+\frac{\tau+1}{2},\tau)&=
(-1)^{k\ell }(-\i)^{1-\ell }\exp(-\pi\i(\frac{\tau}{4}+y))
                                   \h_{1-k,1-\ell }(y,\tau),
\end{align}
where $p,q\in \Z$.
The restriction $\h_{k,\ell}(\tau)=\h_{k,\ell}(0,\tau)$ 
of $\h_{k,\ell}(y,\tau)$ to $y=0$ is called a theta constant,  
which is a holomorphic function on the upper half space $\H$. 
By 
the second of (\ref{eq:basic-theta}),
we have $\h_{1,1}(\tau)=0$. 
It is known that the other theta constants satisfy
Jacobi's identity: 
$$\h_{0,1}(\tau)^4+\h_{1,0}(\tau)^4=\h_{0,0}(\tau)^4.$$

Since the functions 
$$\frac{\h_{1,1}(y+1/2,\tau)}{\h_{0,0}(y+1/2,\tau)}=
-\frac{\h_{1,0}(y,\tau)}{\h_{0,1}(y,\tau)},\quad 
\frac{\h_{1,1}(y+\tau/2,\tau)}{\h_{0,0}(y+\tau/2,\tau)}=
-\i\frac{\h_{0,1}(y,\tau)}{\h_{1,0}(y,\tau)}
$$
are even with respect to $y$, the partial derivative 
$$\frac{\pa }{\pa y}\frac{\h_{1,1}(y,\tau)}{\h_{0,0}(y,\tau)}$$
vanishes at $y=1/2,\tau/2$.

Recall that 
$$\H\ni \tau(x)=-\frac{f_1(x)} {f_2(x)} -\i=\int_{\a_i}\f_i=
\frac{\int_{\a_i}\eta_i}{\int_{\b_i}\eta_i}\quad (i=1,2),
$$
\begin{align*}
T_\tau\oplus T_\tau\ni 
y(x)&=(y_1(x),y_2(x))\\
&=
\big( \frac{1-\i}{4}\cdot \frac{f_3(x)-f_4(x)}{f_2(x)}+\frac{1}{2},
 \frac{1+\i}{4}\cdot \frac{f_3(x)+f_4(x)}{f_2(x)}-\frac{\i}{2}
\big)\\
&=\jmath_{\L}(P_{1-x_1})=
(\jmath_{\L_1}(P_{1-x_1}),\jmath_{\L_2}(P_{1-x_1}))
=(\int_{P_0}^{P_{1-x_1}} 2\f_1,\int_{P_0}^{P_{1-x_1}} 2\f_2).
\end{align*}
We regard the image of Schwarz's map $\mathbf{f}$ for $\cF_2$ as 
\begin{equation}
\label{eq:image-Schwarz}
(y_1(x),y_2(x),\tau(x))\in \C^2\times \H. 
\end{equation}

\begin{lemma}
\label{lem:v(v-1)/(v-1/z)}
The pull back of $\dfrac{\h_{1,1}(y,\tau(x))^4}{\h_{0,0}(y,\tau(x))^4}$ 
under the map $\jmath_{\L_i}$ $(i=1,2)$ coincides with
$\dfrac{v(v-1)}{v-1/z}$, where the coordinate $v$ of $C_z$ is regarded as 
a meromorphic function on $C_z$.
\end{lemma}

\begin{proof}
By Proposition \ref{prop:images}, we have 
$$\Big[\jmath_{\L_1}^*\big (\dfrac{\h_{1,1}(y,\tau)^4}{\h_{0,0}(y,\tau)^4}\big)
\Big]=4P_0+4P_1-4P_{1/z}-4P_\infty,
$$
where $[f]$ denotes the divisor of a meromorphic function $f$.
Thus there exists a non-zero constant $\kappa$ such that 
$\dfrac{v(v-1)}{v-1/z}$ coincides with $\kappa$-times of this function.  
Since $\dfrac{\pa}{\pa y}\dfrac{\h_{1,1}(y,\tau)}{\h_{0,0}(y,\tau)}$ vanishes 
at $y=1/2,\tau/2$, 
the values 
$$\kappa\cdot \dfrac{\h_{1,1}(1/2,\tau)^4}{\h_{0,0}(1/2,\tau)^4}=
\kappa\cdot\dfrac{\h_{1,0}(0,\tau)^4}{\h_{0,1}(0,\tau)^4},\quad 
\kappa\cdot\dfrac{\h_{1,1}(\tau/2,\tau)^4}{\h_{0,0}(\tau/2,\tau)^4}=
\kappa\cdot\dfrac{\h_{0,1}(0,\tau)^4}{\h_{1,0}(0,\tau)^4}
$$
are branched values of $\dfrac{v(v-1)}{v-1/z}$.
Let $\l$ be a branched value of $\dfrac{v(v-1)}{v-1/z}$.
Then the equation $\dfrac{v(v-1)}{v-1/z}=\l$ of $v$ 
has a multiple root. Thus $\l$ satisfies 
\begin{equation}
\label{eq:branch-equation}
(\l+1)^2-\frac{4\l}{z}=\l^2+(2-\frac{4}{z})\l+1=0.
\end{equation}
Hence the product of the branched values should be $1$, and we have 
$$
\Big(\kappa\cdot\dfrac{\h_{1,0}(0,\tau)^4}{\h_{0,1}(0,\tau)^4}\Big)\cdot 
\Big(\kappa\cdot\dfrac{\h_{0,1}(0,\tau)^4}{\h_{1,0}(0,\tau)^4}\Big)=
\kappa^2=1.
$$ 
If $0<x_1,x_2<1$, $x_1+x_2<1$, $0<v<1$ then 
$z$ belongs to the open interval $(0,1)$ and  
$\tau(y)$ becomes pure imaginary by Corollary \ref{cor:tau-pure-im}.
Then  we have
$$\h_{0,1}(0,\tau),\h_{1,0}(0,\tau)>0$$
and 
$$0<\frac{4}{z}-2=\kappa\big( 
\dfrac{\h_{1,0}(0,\tau)^4}{\h_{0,1}(0,\tau)^4}+
\dfrac{\h_{0,1}(0,\tau)^4}{\h_{1,0}(0,\tau)^4}
\big),
$$
hence $\kappa$ should be $1$.
\end{proof}

\begin{remark}
The branched values of $\dfrac{v(v-1)}{v-1/z}$ are
$$\frac{2-z\pm2\sqrt{1-z}}{z}=zv_{\pm}^2,
$$
where $v_{\pm}$ are given in Corollary \ref{cor:ram-pts}.
\end{remark}

By Lemma \ref{lem:v(v-1)/(v-1/z)}, we have 
$$\frac{\h_{1,1}(y_1(x),\tau(x))^4}{\h_{0,0}(y_1(x),\tau(x))^4}
=\frac{\h_{1,1}(y_2(x),\tau(x))^4}{\h_{0,0}(y_2(x),\tau(x))^4}.
$$
We can determine the irreducible component containing 
the image of Schwarz's map $\mathbf{f}$ for $\cF_2$. 

\begin{theorem}
\label{th:image-Schwarz}
The image of Schwarz's map $\mathbf{f}$ for $\cF_2$  
is contained in the analytic set
$$\h_{0,0}(y_1(x),\tau(x))\h_{1,1}(y_2(x),\tau(x))=
\i\h_{1,1}(y_1(x),\tau(x))\h_{0,0}(y_2(x),\tau(x)).
$$
\end{theorem}

\begin{proof}
We check that this analytic set is well defined.
Lemma \ref{lem:Lambda-index2} implies that 
$(y_1(x),y_2(x))$ is transformed into 
$(y_1(x)+p_1\tau+q_1,y_2(x)+p_2\tau+q_2)$ 
by the analytic continuation of $\jmath_{\L}$ along a loop in $C_z$, 
where 
$p_1,q_1,p_2,q_2\in \Z$ satisfy $p_1+q_1+p_2+q_2\in 2\Z$.
By the first formula (\ref{eq:basic-theta}), this set is invariant 
under this continuation.
Under the actions of generators 
$$
\begin{pmatrix}
1 & 2\\
0 & 1
\end{pmatrix},\quad 
\begin{pmatrix}
0 & -1\\
1 & 0
\end{pmatrix}
$$
of the Igusa group $\G_{1,2}$, 
we have
\begin{align*}
&\h_{0,0}(y,\tau+2)=\h_{0,0}(y,\tau),\\
&\h_{1,1}(y,\tau+2)=\i\h_{1,1}(y,\tau),\\
&\h_{0,0}(y/\tau,-1/\tau)=
\sqrt{-\i\tau}\exp(\pi\i y^2/\tau)\h_{0,0}(y,\tau),\\
&\h_{1,1}(y/\tau,-1/\tau)=
-\i\sqrt{-\i\tau}\exp(\pi\i y^2/\tau)\h_{1,1}(y,\tau).
\end{align*}
This analytic set is also invariant under these actions.
By putting 
$$\frac{\h_{1,1}(y_1(x),\tau(x))}{\h_{1,1}(y_2(x),\tau(x))}=
c\frac{\h_{0,0}(y_1(x),\tau(x))}{\h_{0,0}(y_2(x),\tau(x))},
\quad (c=\pm1,\pm\i)
$$
we consider the limit $P\to P_0$.
Since $y_1(x)$ and $y_2(x)$ converge to $0$, the right hand side 
converges to $c$. On the other hand, the left hand side  is indefinite 
of type $0/0$.  
We have 
\begin{align*}
& \lim_{P\to P_0} \frac{\h_{1,1}(\int_{P_0}^P \f_1,\tau(x))}
{\h_{1,1}(\int_{P_0}^P \f_2,\tau(x))}
=\lim_{P\to P_0}\frac{\h_{1,1}'(\int_{P_0}^P \f_1,\tau(x))}
{\h_{1,1}'(\int_{P_0}^P \f_2,\tau(x))}
\frac{\f_1(P)}{\f_2(P)}=\frac{\f_1(P_0)}{\f_2(P_0)}\\
=&\frac{\eta_1(P_0)/\int_{\b_1}\eta_1+ \i\eta_2(P_0)/\int_{\b_2}\eta_2} 
{\i\eta_1(P_0)/\int_{\b_1}\eta_1+ \eta_2(P_0)/\int_{\b_2}\eta_2} 
=
\frac{\eta_1(P_0)/\int_{\b_1}\eta_1} 
{\i\eta_1(P_0)/\int_{\b_1}\eta_1} =-\i
\end{align*}
by l'H\^{o}pital's rule, 
the identity in Proof of Corollary \ref{cor:ex-dual-basis} 
and $\eta_2(P_0)=0$ shown in Proposition \ref{prop:sigma-action}.
\end{proof}

\begin{theorem}
\label{th:z-express}
For $\tau=\tau(x)\in \H$, we have
$$z=\frac{1-x_1-x_2}{(1-x_1)(1-x_2)}
=\frac{4\h_{0,1}(0,\tau)^4\h_{1,0}(0,\tau)^4}{\h_{0,0}(0,\tau)^8}.
$$
\end{theorem}
\begin{proof}
By Proof of Lemma \ref{lem:v(v-1)/(v-1/z)}, the branched values of 
$\dfrac{v(v-1)}{v-1/z}$ are 
$$\frac{\h_{0,1}(0,\tau)^4}{\h_{1,0}(0,\tau)^4},\quad 
  \frac{\h_{1,0}(0,\tau)^4}{\h_{0,1}(0,\tau)^4}.$$
By (\ref{eq:branch-equation}), we have 
$$\frac{\h_{0,1}(0,\tau)^4}{\h_{1,0}(0,\tau)^4}+
\frac{\h_{1,0}(0,\tau)^4}{\h_{0,1}(0,\tau)^4}=\frac{4}{z}-2,
$$
$$z=4\Big/
\Big(\frac{\h_{0,1}(0,\tau)^2}{\h_{1,0}(0,\tau)^2}+
\frac{\h_{1,0}(0,\tau)^2}{\h_{0,1}(0,\tau)^2}\Big)^2
=\frac{4\h_{0,1}(0,\tau)^4\h_{1,0}(0,\tau)^4}
{(\h_{0,1}(0,\tau)^4+\h_{1,0}(0,\tau)^4)^2}.
$$
Use Jacobi's identity.
\end{proof}

\begin{theorem}
\label{th:(1-v)-theta-exp}
For $\tau=\tau(x)\in \H$, we have 
\begin{equation}
\label{eq:(1-v)-theta}
1-v=\frac{\h_{0,0}(\tau)^4}{4\h_{0,1}(\tau)^2\h_{1,0}(\tau)^2}
\Big(
\frac{\h_{0,1}(y_1,\tau)\h_{1,0}(y_1,\tau)}{\h_{0,0}(y_1,\tau)^2}+
\frac{\h_{0,1}(y_2,\tau)\h_{1,0}(y_2,\tau)}{\h_{0,0}(y_2,\tau)^2}
\Big)^2,
\end{equation}
$$y_1=(1-\sigma^2)\cdot \int_{P_0}^P \f_1=\int_{P_0}^P 2\f_1,
\quad 
y_2=(1-\sigma^2)\cdot\int_{P_0}^P \f_2=\int_{P_0}^P 2\f_2,
$$
as meromorphic functions on $C_z$.
\end{theorem}
\begin{proof}
At first, we show that the right hand side (\ref{eq:(1-v)-theta}) is single 
valued on $C_z$.
If we make analytic continuations of 
$y_1,y_2$ along a cycle in $C_z$, then they become 
$$y_1+p_1\tau+q_1,\quad  y_2+p_2\tau+q_2, 
$$ 
$p_1,p_2,q_1,q_2\in \Z$, $p_1+p_2+q_1+q_2\in 2\Z$  
by Lemma \ref{lem:Lambda-index2}.
This analytic continuation transforms
$$\Theta(P)=
\frac{\h_{0,1}(y_1,\tau)\h_{1,0}(y_1,\tau)}{\h_{0,0}(y_1,\tau)^2}+
\frac{\h_{0,1}(y_2,\tau)\h_{1,0}(y_2,\tau)}{\h_{0,0}(y_2,\tau)^2}
$$
into 
$$
(-1)^{p_1+q_1}\frac{\h_{0,1}(y_1,\tau)\h_{1,0}(y_1,\tau)}{\h_{0,0}(y_1,\tau)^2}+
(-1)^{p_2+q_2}\frac{\h_{0,1}(y_2,\tau)\h_{1,0}(y_2,\tau)}{\h_{0,0}(y_2,\tau)^2}
$$
by the first of (\ref{eq:basic-theta}).
Its square is invariant under this analytic continuation 
since $p_1+q_1\equiv p_2+q_2$ $\bmod 2$. 
Next, we show that $\Theta(P)^2$
becomes zero of order $4$ at $P_1$, and  pole of order $4$ at $P_\infty$, 
and it does not have other zeros and poles. 
By Proposition \ref{prop:images} and the first of (\ref{eq:basic-theta}), 
we have 
$$\Theta(P_1)=\frac{\h_{0,1}(1,\tau)\h_{1,0}(1,\tau)}
{\h_{0,0}(1,\tau)^2}+
\frac{\h_{0,1}(0,\tau)\h_{1,0}(0,\tau)}
{\h_{0,0}(0,\tau)^2}
=0.
$$
Since $\h_{0,0}(y,\tau)$, $\h_{0,1}(y,\tau)$ and $\h_{1,0}(y,\tau)$ 
are even functions with respect to $y$, 
the order of zero of $\Theta(P)$ at $P_1$ is greater than or equal to $2$.  
It is well known that $\h_{0,0}(y,\tau)$ vanishes only at 
$\dfrac{\tau+1}{2}$ of first order modulo the lattice $\tau\Z+\Z$.  
Thus the denominators of 
$\dfrac{\h_{0,1}(y_1,\tau)\h_{1,0}(y_1,\tau)}{\h_{0,0}(y_1,\tau)^2}$ and
$\dfrac{\h_{0,1}(y_2,\tau)\h_{1,0}(y_2,\tau)}{\h_{0,0}(y_2,\tau)^2}$
vanish only at $P_{1/z}$ and $P_\infty$ by 
Proposition \ref{prop:images}.
By Theorem \ref{th:image-Schwarz}, we have 
\begin{align*}
\Theta(P)=&
\frac{\h_{0,1}(y_1,\tau)\h_{1,0}(y_1,\tau)}{\h_{0,0}(y_1,\tau)^2}-
\frac{\h_{0,1}(y_2,\tau)\h_{1,0}(y_2,\tau)\h_{1,1}(y_1,\tau)^2}
{\h_{0,0}(y_1,\tau)^2\h_{1,1}(y_2,\tau)^2}\\
=&
\frac{\h_{0,1}(y_1,\tau)\h_{1,0}(y_1,\tau)\h_{1,1}(y_2,\tau)^2
-\h_{0,1}(y_2,\tau)\h_{1,0}(y_2,\tau)\h_{1,1}(y_1,\tau)^2}
{\h_{0,0}(y_1,\tau)^2\h_{1,1}(y_2,\tau)^2}.
\end{align*}
Its numerator vanishes at $P_{1/z}$ since  
$y_1$ and $y_2$ become 
$\dfrac{\tau+1}{2}$ by Proposition \ref{prop:images}.   
Since $\h_{1,1}(y+\dfrac{\tau+1}{2},\tau)$, 
$\h_{0,1}(y+\dfrac{\tau+1}{2},\tau)$ and $\h_{1,0}(y+\dfrac{\tau+1}{2},\tau)$ 
are even functions with respect to $y$, 
the numerator of $\Theta(P)$ vanishes at $P_{1/z}$ 
of order greater than or equal to $2$, and it cancels 
the zero of the denominator of $\Theta(P)$ at $P_{1/z}$.  
Hence $P_{1/z}$ is not a pole of $\Theta(P)$.
On the other hand, the numerator of $\Theta(P)$ does not vanish at $P_{\infty}$
by the first of (\ref{eq:basic-theta}) 
since  
$y_1$ and $y_2$ become $\dfrac{\tau+1}{2}$ and $\dfrac{\tau-1}{2}$
at $P_{\infty}$, respectively, by Proposition \ref{prop:images}.   
Hence, $\Theta(P)^2$ admits poles only on $P_{\infty}$ of order $4$, 
and it vanishes only on $P_1$ of order $4.$ 

Since the divisor of $\Theta(P)^2$ coincides with 
that of the meromorphic function $1-v$ on $C_z$, 
$1-v$ is equal to a constant multiple of $\Theta(P)^2$.
By putting $P=P_0$, we can determine this constant. 
\end{proof}

\begin{theorem}
\label{th:inv-period}
The inverse of Schwarz's map $\mathbf{f}$ for $\cF_2$ is expressed as 
\begin{align*}
x_1=&
\frac{\h_{0,0}(0,\tau(x))^4}{4\h_{0,1}(0,\tau(x))^2\h_{1,0}(0,\tau(x))^2}\\
&\cdot 
\Big(
\frac{\h_{0,1}(y_1(x),\tau(x))\h_{1,0}(y_1(x),\tau(x))}
{\h_{0,0}(y_1(x),\tau(x))^2}
+\frac{\h_{0,1}(y_2(x),\tau(x))\h_{1,0}(y_2(x),\tau(x))}
{\h_{0,0}(y_2(x),\tau(x))^2}\Big)^2,
\\
x_2=&
\frac{\h_{0,0}(0,\tau(x))^4}{4\h_{0,1}(0,\tau(x))^2\h_{1,0}(0,\tau(x))^2}\\
&\cdot 
\Big(
-\frac{\h_{0,1}(y_1(x),\tau(x))\h_{1,0}(y_1(x),\tau(x))}
{\h_{0,0}(y_1(x),\tau(x))^2}
+\frac{\h_{0,1}(y_2(x),\tau(x))\h_{1,0}(y_2(x),\tau(x))}
{\h_{0,0}(y_2(x),\tau(x))^2}\Big)^2, 
\end{align*}
where $(y_1(x),y_2(x),\tau(x))\in \C^2\times \H$ is the image of 
$x=(x_1,x_2)$ under Schwarz's map $\mathbf{f}$ for $\cF_2$ as 
in \eqref{eq:image-Schwarz}.
\end{theorem}

\begin{proof}
We have only to put $P=P_{1-x_1}$ for the expression \eqref{eq:(1-v)-theta}
to obtain the expression of $x_1$.  
Let $x'$ be $(x_2,x_1)$ for $x=(x_1,x_2)$.  
Then we have 
$$y_1(x')=-y_1(x)+1,\quad y_2(x')=y_2(x)
$$
by \eqref{eq:P(1-x1)-express} and \eqref{eq:P(1-x2)-express}.
Thus the expression of $x_2$ is obtained from  
the expression of $x_1$ together with the first and second of 
\eqref{eq:basic-theta}.
\end{proof}

\begin{appendix}
\section{The rational function field $\C(v,w)$ of $C_z$}
In appendix, we study the structure of 
the rational function field $\C(v,w)$ of $C_z:w^4=v^3(1-v)(1-vz)$.
By using results in appendix, 
we explain our construction of the expression \eqref{eq:(1-v)-theta},
which gives an alternative proof of Theorem \ref{th:(1-v)-theta-exp}.

Here recall that 
$$z=\frac{1-x_1-x_2}{(1-x_1)(1-x_2)}$$
for $(x_1,x_2)\in X$, and  
fix $(x_1,x_2)\in X$ satisfying \eqref{eq:assum-x1-x2}.
Since $0<z<1$, we assign 
$\sqrt{z}$, $\sqrt[4]{z}$ and $\sqrt{1-z}$ positive real values.

\subsection{An involution $\iota$}
The automorphism $\sigma:C_z\ni (v,w)\mapsto (v,\i w)$ acts naturally 
on the field $\C(v,w)$ of rational functions on $C_z$.
We define an involution $\iota$ on $\C(v)$ by 
$$\iota(v)=\frac{1-v}{1-vz}.$$

This involution is extended to that on $\C(v,w)$
by 
$$\iota(w)=\frac{v(1-v)\sqrt{1-z}}{(1-vz)w}.$$
This extended involution induces an involution of the curve $C_z$ itself, 
which is also denoted by $\iota$. 
By a straight forward calculation, we can see that 
$\sigma$ and $\iota$ satisfy 
relations
$$
\sigma^4=\iota^2=\mathrm{id},\quad \iota\sigma =\sigma^{-1} \iota,
$$
where $\mathrm{id}$ denotes the identity.
Thus the group generated by $\sigma$ and $\iota$ is isomorphic to 
the dihedral group $D_4=\la \sigma \ra \rtimes\la \iota\ra$ of order $8$.
The fixed field $\C(v)^\iota$ of $\C(v)$ under $\iota$ is 
$$\C(\frac{v(1-v)}{1-vz}).$$
By setting $u=w^2$ and $t=\dfrac{v(1-v)}{1-vz}$, we have the following 
inclusions and the corresponding projections between algebraic curves:
\begin{equation}
\label{eq:inclusions}
\begin{array}{ccl}
\C(v,w) & \supset & \C(v,w)^\iota \\
\cup    &         & \hspace{5mm}\cup \\
\C(u,v) & \supset & \C(u,v)^\iota \\
\cup    &         & \hspace{5mm}\cup \\
\C(v) & \supset & \C(v)^\iota=\C(t), \\   
  \end{array}
\quad 
\begin{array}{ccl}
C_z & \overset{\pr_{\iota}}{\to} & C_z/\la\iota\ra\simeq T_\tau  \\
\downarrow_{\pr_{wu}}    &         & \hspace{5mm}\downarrow_{\pr_{s}} \\
C_z/\la \sigma^2\ra=E_z  & \overset{\pr_{us}}\to & C_z/\la \sigma^2,\iota \ra 
\simeq E_z/\la \iota \ra\\
\downarrow_{\pr_{uv}}    &         & \hspace{5mm}\downarrow_{\pr_{st}} \\
C_z/\la \sigma\ra\simeq\P^1_v & \overset{\pr_{vt}}{\to} 
& C_z/\la \sigma,\iota \ra\simeq 
\P^1_t. \\   
  \end{array}
\end{equation}
Here note that $u,v$ are invariant under the action $\sigma^2$, and that  
the quotient space $E_z=C_z/\la \sigma^2\ra$ is represented by an 
algebraic curve 
\begin{equation}
\label{eq:def-E}
u^2=v^3(1-v)(1-vz),
\end{equation}
which is equivalent to a non-singular cubic curve 
$$
{u'}^2=v(1-v)(1-vz)
$$
with respect to variables $v$ and $u'=(u/v)$. 

By analyzing the branch locus, we find a generator of the field  
$\C(u,v)^\iota$ over $\C$. 
\begin{lemma}
\label{lem:vars-fix}
The involution $\iota$ fixes an element 
$\dfrac{u}{v(1-vz)}\in \C(u,v)$, i.e.,  
$$\iota\big(\frac{u}{v(1-vz)}\big)=\frac{u}{v(1-vz)}.
$$
\end{lemma}
\begin{proof}
Since 
$$
\iota(1-vz)=1-z\frac{1-v}{1-vz}=\frac{1-z}{1-vz},
$$
we have 
$$\iota(\frac{u}{v(1-vz)})=\iota(\frac{w^2}{v}\cdot 
\frac{1}{1-vz})=
\big(\frac{v(1-v)\sqrt{1-z}}{(1-vz)w}\big)^2\cdot\frac{1-vz}{1-v}\cdot 
\frac{1-vz}{1-z} 
=\frac{v^2(1-v)}{u}.
$$
We see that the last term is equal to 
$\dfrac{u}{v(1-vz)}$  
by using a relation $1-v=\dfrac{u^2}{v^3(1-vz)}$ equivalent to 
\eqref{eq:def-E}. 
\end{proof}
We set 
\begin{equation}
\label{eq:def-s}
s=\frac{u}{v(1-vz)}=\frac{v^2(1-v)}{u}=\frac{w^2}{v(1-vz)}=\frac{v^2(1-v)}{w^2}.
\end{equation}

\begin{proposition}
\label{prop:gen-s}
We have 
$$\C(u,v)^\iota=\C(s).$$
\end{proposition}
\begin{proof}
Note that $\C(u,v)^\iota$ is a quadratic extension 
of the field $\C(v)^\iota=\C(t)$. 
We have 
$$s^2=\big(\frac{u}{v(1-vz)}\big)^2=\frac{v(1-v)}{1-vz}=t,
$$
and $s$ belongs to $\C(u,v)^\iota$ by Lemma \ref{lem:vars-fix}.
Hence the field $\C(u,v)^\iota$ is generated by $s$ over $\C$.
\end{proof}

Let $\P^1_s$, $\P^1_t$ and $\P^1_v$  be projective spaces with $s$-coordinate
, with $t$-coordinate, and with $v$-coordinate, respectively. 
Then $\P^1_s$ is a double cover of $\P^1_t$ 
by Proposition \ref{prop:gen-s} and  
the branch locus of the projection $\P^1_s\to \P^1_t$
is $t=0,\infty$. 
Let $v_+$ and $v_-$ be the ramification points of the projection 
$$\pr_{vt}:\P^1_v\ni v \mapsto t=\frac{v(1-v)}{1-vz}\in \P^1_t.$$ 
In other words, $v_{\pm}$ are solutions of the equation 
$$v=\iota(v)=\frac{1-v}{1-vz},
$$
and they are expressed as
\begin{equation}
\label{eq:ramify-v}
v_\pm=\frac{1\pm\sqrt{1-z}}{z}.
\end{equation}

The branch points of $\P^1_v\to \P^1_t=\P^1_v/\la \iota \ra$ are given by 
$$t_\pm=\frac{v_\pm(1-v_\pm)}{1-v_\pm z}=v_\pm\cdot \pr_{vt}(v_{\pm})
=(v_\pm)^2.
$$
The rational function field of $E_z$ and that of $E_z/\la \iota\ra$
are isomorphic to $\C(u,v)$ and $\C(s)=\C(u,v)^\iota$, respectively.
The set of preimages of $t_\pm$ under the projection 
$$\pr_{st}:\P^1_s=E/\la \iota\ra\ni s \mapsto t=s^2\in 
\P^1_s/\la \iota\ra=\P^1_t$$ 
is given by $\{\pm v_+,\pm v_-\}$. 

We find branch points of the projection 
$$\pr_{s}: C_z/\la \iota\ra \to E_z/\la \iota\ra=\P^1_s.$$
An equation of fixed points of $\iota$ is given by $\iota(w)=w$, 
which is equivalent to 
\begin{equation}
\label{eq:iota-fix}
\frac{v(1-v)\sqrt{1-z}}{(1-vz)}=w^2.
\end{equation}
Let $F_\iota$ be the set of fixed points of $\iota$ in $C_z$.
Since $u=w^2$ and \eqref{eq:iota-fix}, 
the image of each element of $F_\iota$ under the 
projection $\pr_{wu}:C_z\to E_z$  is determined by the $v$-coordinate.
In fact, the $v$-coordinate of any fixed point is $v_+$ or $v_-$,  
the image of $F_\iota$ under $\pr_{wu}$ consists of 
\begin{equation}
\label{eq:Cz-fix}
(v,u)=(v_+,u_+),\ (v_-,u_-),\quad 
u_{\pm}=\frac{v_{\pm}(1-v_{\pm})\sqrt{1-z}}{(1-v_{\pm}z)}. 
\end{equation}
Let us consider the branching situation over $v_+$ and $v_-$ by using 
the commutative diagram  
\begin{equation}
\label{eq:CD-projections}
\begin{array}{ccl}
C_z & \overset{\pr_{\iota}}{\to} & C_z/\la\iota\ra\simeq T_\tau  \\
\downarrow_{\pr_{wu}}    &         & \hspace{5mm}\downarrow_{\pr_{s}} \\
C_z/\la \sigma^2\ra=E_z  & \overset{\pr_{us}}\to & C_z/\la \sigma^2,\iota \ra 
\simeq E_z/\la \iota \ra\simeq\P^1_s.\\
\end{array}
\end{equation}  
in \eqref{eq:inclusions}. 
As in Corollary \ref{cor:preimages}, 
$P_{v_+}$, $\sigma^2(P_{v_+})$ and  $\sigma(P_{v_-})$, $\sigma^3(P_{v_-})$
are not fixed by the action $\iota$. 
Since the images of fixed points by $\iota$ under the projection $\pr_{wu}$ 
are given by $(u_+,v_+)$, $(u_-,v_-)$, the images of 
$P_{v_+},\sigma^2(P_{v_+})$ and $\sigma(P_{v_-})$, $\sigma^3(P_{v_-})$  
under the projection $\pr_{wu}$ 
are 
$$
\pr_{wu}(P_{v_+})=\pr_{wu}(\sigma^2(P_{v_+}))=(-u_+,v_+),\quad 
\pr_{wu}(\sigma(P_{v_-}))=\pr_{wu}(\sigma^3(P_{v_-}))=(-u_-,v_-).
$$
Hence the $s$-coordinate of the images of 
$P_{v_+}$, $\sigma^2(P_{v_+})$ and $\sigma(P_{v_-})$, $\sigma^3(P_{v_-})$ 
under $\pr_{us}\circ\pr_{wu}$ 
are 
\begin{equation}
\label{eq:s-coord-branch}
s_{\pm}=\frac{-u_{\pm}}{v_{\pm}(1-v_{\pm}z)}
\end{equation}
by \eqref{eq:def-s}.

\begin{lemma}
\label{lem:ramify-prs}
The projection 
$\pr_s:C_z/\la \iota \ra\to 
C_z/\la \sigma^2,\iota \ra
\simeq \P^1_s$ is ramified at $\pr_\iota(P_{v_+})$,
$\pr_\iota(\sigma^2(P_{v_+}))$
and $\pr_\iota(\sigma(P_{v_-}))$, $\pr_\iota(\sigma^3(P_{v_-}))$.
\end{lemma}
\begin{proof}
As mentioned above, the projection $\pr_{wu}$  is not ramified at 
$\sigma^{2i}(P_{v_+})$ $(i=0,1)$, and  
$\sigma^{2i}(P_{v_+})$ are mapped to the point 
$(-u_+,v_+)\in E_z$ under the projection 
$\pr_{wu}:C_z\to C_z/\la \sigma^2\ra= E_z$. 
Since the point $(-u_+,v_+)\in E_z$
is fixed by the action $\iota$, the projection 
$\pr_{us}:E_a\to C_z/\la \sigma^2,\iota\ra\simeq \P^1_s$ 
is ramified at these point. By the commutative diagram 
\eqref{eq:CD-projections}, it turns out that the projection $\pr_s$ 
is ramified at $\sigma^{2i}(P_{v_+})$. We can similarly show that 
the projection $\pr_s$ is ramified at $\sigma^{2i+1}(P_{v_-})$ $(i=0,1)$.
\end{proof}

\begin{proposition}
\label{prop:Cz/iota}
The branch points of the projection under 
$\pr_{s}:C_z/\la \iota\ra\to 
C_z/\la \sigma^2,\iota\ra\simeq \P^1_s
$  
are expressed as 
$$s=0,\ \infty,\ v_+,\ -v_-,$$
by the coordinate of $\P^1_s$.
\end{proposition}

\begin{proof}

By \eqref{eq:ramify-v} and \eqref{eq:Cz-fix}, 
$s_{+}$ and $s_{-}$ in \eqref{eq:s-coord-branch} 
reduce to $v_+$ and $-v_-$, respectively.
Since $\pr_{wu}$ is ramified at $P_0$ and $P_\infty$, 
they are ramification points of 
$\pr_{us}\circ \pr_{wu}=\pr_{s}\circ \pr_{\iota}$. 
Since the $v$-coordinates of $P_0$ and $P_\infty$ are not fixed by $\iota$, 
$\pr_{\iota}$ is not ramified at $P_0$ and $P_\infty$. 
Thus $\pr_{s}$ ramified at $\pr_{\iota}(P_0)$ and $\pr_{\iota}(P_\infty)$. 
We have 
$$\pr_s(\pr_{\iota}(P_0))=\pr_{us}\circ \pr_{wu}(P_0)=0,\quad 
\pr_s(\pr_{\iota}(P_\infty ))=\pr_{us}\circ \pr_{wu}(P_\infty)=\infty,
$$
by \eqref{eq:def-s} and Table \ref{tab:zeros} below.
\end{proof}

\subsection{The structure of $\C(v,w)^\iota$}
We define rational functions on $C_z$ by 
$$h_+=h_+(s)=\frac{(s-v_+)(s+v_-)}{s},\quad 
h_-=h_-(s)=\frac{(s+v_+)(s-v_-)}{s},
$$
where $s=\dfrac{w^2}{v(1-vz)}$.
We list the orders of zero of the functions $w$, $v$, $w/v$ and $h_{\pm}$  at 
some points of $C_z$ in Table \ref{tab:zeros}, where $i=0,1$.

\begin{table}[htb]
$$
\begin{array}{|c||c|c|c|c|c|c|}
\hline
\textrm{functions\textbackslash points}
& P_0 & P_1 &P_{1/z} &P_\infty & \sigma^{2i}(P_{v_+}),
\sigma^{2i+1}(P_{v_-})&
\sigma^{2i+1}(P_{v_+}),\sigma^{2i}(P_{v_-})\\
\hline
w & 3 & 1 & 1& -5& 0  &0 \\
\hline
v & 4 & 0 & 0& -4& 0  &0 \\
\hline
w/v & -1 & 1 & 1& -1& 0  &0 \\
\hline
h_+ &-2&-2&-2& -2& 2  &0 \\
\hline
h_- &-2&-2&-2& -2& 0  &2 \\
\hline
\end{array}
$$
\caption{Orders of zero of functions}
\label{tab:zeros}
\end{table}
Let $H^0(C_z,\mathcal{O}(D))$ be the vector space of 
rational functions on $C_z$ admitting poles bounded by $D$, 
where $D=P_0+P_1+P_{1/z}+P_\infty$ is an effective  divisor on $C_z$. 
This space is $3$-dimensional by 
the Riemann-Roch theorem and Proposition \ref{prop:sigma-action}.
By Table \ref{tab:zeros}, we see that the functions $1$, $w/v$ and $v/w$ 
span $H^0(C_z,\mathcal{O}(D))$. 
We set 
\begin{equation}
\label{eq:f+-}
f_+=\frac{w}{v}+\sqrt{1-z}\frac{v}{w},\quad 
f_-=\frac{w}{v}-\sqrt{1-z}\frac{v}{w},
\end{equation}
which belong to $H^0(C_z,\mathcal{O}(D))$.
\begin{lemma}
\label{lem:gen-C(v,w)iota}
\begin{enumerate}
\item
\label{lem:inv-sigma-iota-1}
The function $f_+$ is a unique element in $H^0(C_z,\mathcal{O}(D))$ 
satisfying $$\sigma^2(f_+)=-f_+,\quad \iota(f_+)=f_+$$
up to non-zero constant multiplication.
\item
\label{lem:inv-sigma-iota-2}
The function $f_-$ is a unique element in $H^0(C_z,\mathcal{O}(D))$ 
satisfying $$\sigma^2(f_-)=-f_-,\quad (\sigma^2\iota)(f_-)=f_-$$
up to non-zero constant multiplication.
\end{enumerate}
\end{lemma}
\begin{proof} 
(1) Let $f_1$ be an element of $H^0(C_z,\mathcal{O}(D))$ 
satisfying $\sigma^2(f_1)=-f_1$, $\iota(f_1)=f_1$. 
Express $f_1$ as a linear combination 
$$c_1\cdot 1+c_2\cdot (w/v)+c_3\cdot(v/w)$$
of $1,w/v,v/w$.
By the condition $\sigma^2(f_1)=-f_1$, we have 
$$c_1\cdot 1-c_2\cdot (w/v)-c_3\cdot(v/w)=\sigma^2(f_1)=
-f_1=-c_1\cdot 1-c_2\cdot (w/v)-c_3\cdot(v/w),$$
which yields $c_1=0$. 
Since 
$$\iota\big( \frac{w}{v}\big)=
\frac{v(1-v)\sqrt{1-z}}{(1-vz)w}\cdot \frac{1-zv}{1-v}=\sqrt{1-z}\frac{v}{w},
$$
we have 
$$c_2\cdot\sqrt{1-z} (v/w)+c_3\cdot\frac{1}{\sqrt{1-z}}\cdot(w/v)
=\iota(f_1)
=f_1=c_2\cdot (w/v)+c_3\cdot(v/w),
$$
which yields $\sqrt{1-z} c_2=c_3$.
Thus $f_1$ is a non-zero constant multiplication of $f_+$. 

\smallskip\noindent
(2) Use an argument similar to (1). 
\end{proof}

\begin{proposition}
\label{prop:C(v,w)iota}
\begin{enumerate}
\item 
\label{prop:C(v,w)sigma0iota}
The function $f_+$ is a generator of the extension field 
$$\C(v,w)^\iota \supset \C(u,v)^\iota\simeq\C(s).$$

\item 
\label{prop:C(v,w)sigma2iota}
The function $f_-$ is a generator of the extension field 
$$\C(v,w)^{\sigma^2\iota} \supset \C(u,v)^\iota\simeq\C(s).$$
\end{enumerate}
\end{proposition}

\begin{proof}
(1) 
As shown in Lemma \ref{lem:gen-C(v,w)iota}, 
$f_+$ is invariant under $\iota$.
Since it satisfies $\sigma^2(f_\pm)=-f_{\pm}$, 
it becomes a generator of the extension field $\C(v,w)^\iota$ 
over $\C(u,v)^\iota=\C(v,w)^{\la \sigma^2, \iota\ra}$ by Kummer theory.

\smallskip\noindent
(2) Use an argument similar to (1). 
\end{proof}

By Proposition \ref{prop:Cz/iota} and Table \ref{tab:zeros}, 
$C_z/\la \iota\ra$ is isomorphic to an elliptic curve, 
and its rational function field $\C(v,w)^\iota$ 
is isomorphic to the quadratic extension $\C(s)(\sqrt{h_+})$ of $\C(s)$. 
Thus $\sqrt{h_+}$ is regarded as a rational function on  
$C_z/\la \iota\ra$
and its pull back of $\sqrt{h_+}$ 
under the projection $C_z\to C_z/\la \iota\ra$ satisfies
$$\sigma^2(\sqrt{h_+})=-\sqrt{h_+},\quad 
\iota(\sqrt{h_+})=\sqrt{h_+}
$$
as functions of $C_z$.
Hence $h_+$ is equal to $f_+^2$  up to non-zero constant multiplication.
Similarly, $C_z/\la \sigma^2\iota\ra$ is isomorphic to 
an elliptic curve, its rational function field 
$\C(v,w)^{\sigma^2\iota}$ is isomorphic to $\C(s)(\sqrt{h_-})$, 
$\sqrt{h_-}$ satisfies 
$$\sigma^2(\sqrt{h_-})=-\sqrt{h_-},\quad 
\sigma^2\iota(\sqrt{h_-})=\sqrt{h_-},
$$
and 
$h_-$ is equal to $f_-^2$  up to non-zero constant multiplication.
\begin{proposition}
\label{prop:fh-relation}
We have 
\begin{align*}
f_+^2&=-zh_+=-z\frac{(s-v_+)(s+v_-)}{s},\\
f_-^2&=-zh_-=-z\frac{(s+v_+)(s-v_-)}{s}.
\end{align*}
The curves $C_z/\la \iota\ra$ and $C_z/\la \sigma^2\iota\ra$ 
are represented by non-singular cubic curves 
\begin{align*}
E_1:f_+^2s&=-z(s-v_+)(s+v_-),\\
E_2:f_-^2s&=-z(s+v_+)(s-v_-),
\end{align*}
in the variables $f_+,s$ and in $f_-,s$, respectively.
\end{proposition}
\begin{proof}
By straight forward calculations, we have 
\begin{align*}
f_+^2&
=\frac{1}{uv^2}(u+\sqrt{1-z} v^2)^2
=\frac{1}{uv^2}(u^2+2\sqrt{1-z} uv^2+v^4(1-z))\\
&=\frac{1}{uv^2}\left(v^3(1-v)(1-vz)+2\sqrt{1-z} uv^2+v^4(1-z)\right)\\
&=\frac{1}{u}\left(v((1-v)(1-vz) +v(1-z))+2\sqrt{1-z} u\right)\\
&=\frac{1}{u}\left(v(1-2vz+v^2z)+2\sqrt{1-z} u\right),\\
h_+&=\big(\frac{u}{v(1-vz)}-v_+\big)
\big(\frac{u}{v(1-vz)}+v_-\big)\frac{v(1-vz)}{u}\\
&=\frac{(u-v_+v(1-vz))(u+v_-v(1-vz))}{uv(1-vz)}\\
&=\frac{u^2+(v_--v_+)uv(1-vz)-v_-v_+v^2(1-vz)^2}{uv(1-vz)}\\
&=\frac{zu^2-2\sqrt{1-z} uv(1-vz)-v^2(1-vz)^2}{zuv(1-vz)}
=\frac{zv^2(1-v)-2\sqrt{1-z} u-v(1-vz)}{zu}\\
&=\frac{1}{zu}\big(v(v(1-v)z-(1-vz)\big)-2\sqrt{1-z} u\big)
=\frac{1}{zu}\big(v(2vz-v^2z-1)-2\sqrt{1-z} u\big),
\end{align*}
which show the first identity. The second one can be shown similarly.
\end{proof}

Proposition \ref{prop:fh-relation} implies the following corollary. 
\begin{cor}
\label{cor:omega-relation}
Set
$$f'_+=sf_+,\quad f'_-=sf_-.$$
Then $C_z/\la \iota\ra$ and $C_z/\la \sigma^2\iota\ra$ are 
represented by non-singular cubic curves 
\begin{align*}
{f_+'}^2=-zs(s-v_+)(s+v_-),\\
{f_-'}_-^2=-zs(s+v_+)(s-v_-),
\end{align*}
in the variables $s,f'_+$ and in $s,f'_-$, respectively.
\end{cor}
Let $\psi:E_2\to E_1$, 
$\pr_i:C_z\to E_i$ be morphisms 
defined by 
 \begin{align*}
 \psi&:E_2\ni (f_-,s) \mapsto (f_+,s)=(\i f_-,-s)\in E_1,\\
 pr_1&:C_z\ni(v,w)\mapsto (f_+,s)=(\frac{w}{v}+\sqrt{1-z}\frac{v}{w},
 \frac{u}{v(1-vz)})\in E_1,\\
 pr_2&:C_z\ni(v,w)\mapsto (f_-,s)=(\frac{w}{v}-\sqrt{1-z}\frac{v}{w},
 \frac{u}{v(1-vz)})\in E_2.
 \end{align*}

\begin{proposition}
We have the following commutative diagram:
\begin{equation}
\label{eq:diagram}
\begin{array}{ccc}
 C_z & \overset{\pr_2}{\longrightarrow} & E_2\\
 \sigma\downarrow & &\downarrow \psi\\
 C_z & \overset{\pr_1}{\longrightarrow} & E_1,\\
   \end{array}
\end{equation}
where recall that $\sigma:C_z\ni(v,w)\mapsto (v,\i w)\in C_z.$
\end{proposition}
\begin{proof}
Note that
\begin{align*}
\pr_1\circ \sigma(v,w)&=\pr_1(v,\i w)=
(\frac{\i w}{v}+\sqrt{1-z}\frac{v}{\i w},\frac{-u}{v(1-vz)}), \\
\psi \circ \pr_2(v,w)&=
\psi(\frac{w}{v}-\sqrt{1-z}\frac{v}{w},\frac{u}{v(1-vz)})=
(\i\frac{w}{v}-\i\sqrt{1-z}\frac{v}{w},\frac{-u}{v(1-vz)});
\end{align*}
they coincide. 
\end{proof}

\subsection{Expressions of rational functions on $C_z$ in terms of theta functions }

Recall that 
$$\tau=\int_{\a_1}\f_1=\int_{\a_2}\f_2\in \H$$ 
for $\f_1,\f_2\in H^0_{-\sigma^2}(C_z,\W)$ defined in \eqref{eq:dual-Lambda}, 
and 
$$y_1=(1-\sigma^2)\cdot \int_{P_0}^P \f_1=\int_{P_0}^P 2\f_1,\quad 
y_2=(1-\sigma^2)\cdot \int_{P_0}^P \f_2=\int_{P_0}^P 2\f_2$$
for an element $P=(v,w)$ of the curve $C_z$ with fixed
$z=\frac{1-x_1-x_2}{(1-x_1)(1-x_2)}$. Theorem \ref{th:z-express} states that
$$z=\frac{4\h_{0,1}^4(\tau)\h_{1,0}^4(\tau)}{\h_{0,0}^8(\tau)},$$
where $\h_{0,0}(\tau)=\h_{0,0}(0,\tau)$, $\h_{0,1}(\tau)=\h_{0,1}(0,\tau)$ and  
$\h_{1,0}(\tau)=\h_{1,0}(0,\tau)$.
Recall that $\sqrt{z}$ and $\sqrt[4]{z}$ are assigned by positive real values. 
Since $\tau$ is pure imaginary by Corollary \ref{cor:tau-pure-im}, 
$\h_{0,0}(\tau)$, $\h_{0,1}(\tau)$ and $\h_{1,0}(\tau)$ are positive real.
Thus we have
\begin{equation}
\label{eq:sqrt-z}
\sqrt{z}=\frac{2\h_{0,1}(\tau)^2\h_{1,0}(\tau)^2}{\h_{0,0}(\tau)^4}>0,
\quad 
\sqrt[4]{z}=\frac{\sqrt{2}\h_{0,1}(\tau)\h_{1,0}(\tau)}{\h_{0,0}(\tau)^2}>0.
\end{equation}

\begin{lemma}
\label{lem:f+-}
The functions $s$, $f^2_+$ and $f^2_-$ are expressed as 
\begin{align*}
  s(P)&=\frac{1}{\sqrt{z}}\frac{\h_{1,1}(y_1,\tau)^2}{\h_{0,0}(y_1,\tau)^2}
=-\frac{1}{\sqrt{z}}\frac{\h_{1,1}(y_2,\tau)^2}{\h_{0,0}(y_2,\tau)^2},\\
f^2_+(P)&=2\frac{\h_{0,1}(y_1,\tau)^2\h_{1,0}(y_1,\tau)^2}
{\h_{0,0}(\tau)(y_1,\tau)^2\h_{1,1}(y_1,\tau)^2},\quad
f^2_-(P)=-2\frac{\h_{0,1}(y_2,\tau)^2\h_{1,0}(y_2,\tau)^2}
{\h_{0,0}(y_2,\tau)^2\h_{1,1}(y_2,\tau)^2}.
\end{align*}
\end{lemma}

\begin{proof}
By Lemma \ref{lem:v(v-1)/(v-1/z)},  
we have 
$$s(P)^2=\frac{v(1-v)}{1-vz}=
\frac{1}{z}\frac{v(v-1)}{v-1/z}=\frac{1}{z}
\frac{\h_{1,1}(y_1,\tau)^4}{\h_{0,0}(y_1,\tau)^4}
,\quad 
s(P)=\pm\frac{1}{\sqrt{z}}
\frac{\h_{1,1}(y_1,\tau)^2}{\h_{0,0}(y_1,\tau)^2}
$$
for $P=(v,w)\in C_z$. Let us determine the sign for $s(P)$.
By \eqref{eq:def-s} and $v_-\in (0,1)$, $s(P_{v_-})$ is positive and 
$$y_1(P_{v_-})=\int_{P_0}^{P_{v_-}}2\f_1=\frac{1}{2},\quad 
\frac{1}{\sqrt{z}}
\frac{\h_{1,1}(\frac{1}{2},\tau)^2}{\h_{0,0}(\frac{1}{2},\tau)^2}>0,
$$
which yield the first expression. By Theorem \ref{th:image-Schwarz}, 
we have the second expression of $s(P)$.

Since the divisor of $f^2_+$ coincides with that of 
$\frac{\h_{0,1}(y_1,\tau)^2\h_{1,0}(y_1,\tau)^2}
{\h_{0,0}(y_1,\tau)^2\h_{1,1}(y_1,\tau)^2}$ by Table \ref{tab:zeros}, 
there exists non-zero constant $c$ such that 
$f^2_+(P)=c\frac{\h_{0,1}(y_1,\tau)^2\h_{1,0}(y_1,\tau)^2}
{\h_{0,0}(y_1,\tau)^2\h_{1,1}(y_1,\tau)^2}$.
By substituting 
$$s(P)=\frac{1}{\sqrt{z}}
\frac{\h_{1,1}(y_1,\tau)^2}{\h_{0,0}(y_1,\tau)^2},\quad 
f^2_+(P)=c\frac{\h_{0,1}(y_1,\tau)^2\h_{1,0}(y_1,\tau)^2}
{\h_{0,0}(y_1,\tau)^2\h_{1,1}(y_1,\tau)^2}$$
into $f^2_+s=-z(s-v_+)(s+v_-)$ in Proposition \ref{prop:fh-relation} 
and  putting $P=P_0$, we have 
$$c\frac{1}{\sqrt{z}}
\frac{\h_{0,1}(0,\tau)^2\h_{1,0}(0,\tau)^2}{\h_{0,0}(0,\tau)^4}
=zv_+v_-=1 \Leftrightarrow c=2.
$$ 
Similarly, we can obtain the expression of $f^2_-(P)$ by using 
$s(P)
=-\dfrac{1}{\sqrt{z}}\dfrac{\h_{1,1}(y_2,\tau)^2}{\h_{0,0}(y_2,\tau)^2}.
$
\end{proof}

\begin{theorem}
\label{th:1-v}
The function $1-v$ is expressed as
\begin{align*}
1-v&=
\Big(
\frac{\h_{0,0}(\tau)^2\h_{0,1}(y_1,\tau)\h_{0,1}(y_1,\tau)}
{2\h_{0,1}(\tau)\h_{1,0}(\tau)\h_{0,0}(y_1,\tau)^2}
+
\frac{\h_{0,0}(\tau)^2\h_{0,1}(y_2,\tau)\h_{0,1}(y_2,\tau)}
{2\h_{0,1}(\tau)\h_{1,0}(\tau)\h_{0,0}(y_2,\tau)^2}
\Big)^2\\
&=
\Big(\frac{
\h_{0,0}(\tau)^2\h_{0,1}(y_1,\tau)\h_{0,1}(y_1,\tau)
-\i\h_{0,0}(\tau)^2\h_{0,1}(y'_1,\tau)\h_{0,1}(y'_1,\tau)}
{2\h_{0,1}(\tau)\h_{1,0}(\tau)\h_{0,0}(y_1,\tau)^2}\Big)^2
,
\end{align*}
where 
$\ds{y_1'=\int_{P_0}^{\sigma(P)} 2\f_1}$.
\end{theorem}

\begin{proof}
By \eqref{eq:def-s} and \eqref{eq:f+-}, we have
$$\frac{f_++f_-}{2}=\frac{w}{v},\quad 
 s\cdot\frac{w^2}{v^2}=\frac{w^4}{v^3(1-vz)}=1-v.
$$
Lemma \ref{lem:f+-} yields that 
$$f_+(P) =\pm \sqrt{2}\frac{\h_{0,1}(y_1,\tau)\h_{1,0}(y_1,\tau)}
{\h_{0,0}(y_1,\tau)\h_{1,1}(y_1,\tau)},\quad
f_-(P)=\pm \sqrt{2}\i\frac{\h_{0,1}(y_2,\tau)\h_{1,0}(y_2,\tau)}
{\h_{0,0}(y_2,\tau)\h_{1,1}(y_2,\tau)}.
$$
Thus we have 
\begin{align*}
1-v=&s(P)\frac{(f_+(P)+f_-(P))^2}{4}\\
=&
\frac{1}{2\sqrt{z}}
\Big(\frac{\h_{0,1}(y_1,\tau)\h_{1,0}(y_1,\tau)}
{\h_{0,0}(y_1,\tau)\h_{1,1}(y_1,\tau)}
\frac{\h_{1,1}(y_1,\tau)}{\h_{0,0}(y_1,\tau)}
\pm\i 
\frac{\h_{0,1}(y_2,\tau)\h_{1,0}(y_2,\tau)}
{\h_{0,0}(y_2,\tau)\h_{1,1}(y_2,\tau)}
\frac{-\i \h_{1,1}(y_2,\tau)}{\h_{0,0}(y_2,\tau)}
\Big)^2\\
=&\frac{\h_{0,0}(\tau)^4}{4\h_{0,1}(\tau)^2\h_{1,0}(\tau)^2}
\Big(\frac{\h_{0,1}(y_1,\tau)\h_{1,0}(y_1,\tau)}
{\h_{0,0}(y_1,\tau)^2}
\pm
\frac{\h_{0,1}(y_2,\tau)\h_{1,0}(y_2,\tau)}
{\h_{0,0}(y_2,\tau)^2}
\Big)^2\\
=&\Big(
\frac{\h_{0,0}(\tau)^2\h_{0,1}(y_1,\tau)\h_{0,1}(y_1,\tau)}
{2\h_{0,1}(\tau)\h_{1,0}(\tau)\h_{0,0}(y_1,\tau)^2}
\pm
\frac{\h_{0,0}(\tau)^2\h_{0,1}(y_2,\tau)\h_{0,1}(y_2,\tau)}
{2\h_{0,1}(\tau)\h_{1,0}(\tau)\h_{0,0}(y_2,\tau)^2}
\Big)^2
\end{align*}
by Theorem \ref{th:image-Schwarz}.
Let us determine the sign for these expressions.
Since $1-v$ does not vanish at $P=P_0$ and 
$$\frac{\h_{0,1}(y_1,\tau)\h_{1,0}(y_1,\tau)}
{\h_{0,0}(y_1,\tau)^2}
-
\frac{\h_{0,1}(y_2,\tau)\h_{1,0}(y_2,\tau)}
{\h_{0,0}(y_2,\tau)^2}
$$
vanishes at $P=P_0$, we should select ``$+$''.
Hence we have the first expression of $1-v$.
Since $s(\sigma(P))=-s(P)$ by \eqref{eq:def-s}, and 
$f_+(\sigma(P))=-\i f_-(P)$  by the diagram \eqref{eq:diagram}, 
we have the second expression of $1-v$.
\end{proof}

\begin{cor}
\label{th:w/v}
The rational functions $w^2/v^2$  and $w/v$ on $C_z$ are expressed as 
\begin{align*}
\frac{w^2}{v^2}
=&\frac{1}{2}
\Big(\frac{\h_{0,1}(y_1,\tau)\h_{1,0}(y_1,\tau)}
{\h_{0,0}(y_1,\tau)\h_{1,1}(y_1,\tau)}
+\i 
\frac{\h_{0,1}(y_2,\tau)\h_{1,0}(y_2,\tau)}
{\h_{0,0}(y_2,\tau)\h_{1,1}(y_2,\tau)}
\Big)^2,
\\
\frac{w}{v}=&\frac{-1}{\sqrt{2}}
\Big(\frac{\h_{0,1}(y_1,\tau)\h_{1,0}(y_1,\tau)}
{\h_{0,0}(y_1,\tau)\h_{1,1}(y_1,\tau)}
+\i 
\frac{\h_{0,1}(y_2,\tau)\h_{1,0}(y_2,\tau)}
{\h_{0,0}(y_2,\tau)\h_{1,1}(y_2,\tau)}
\Big). 
\end{align*}
\end{cor}

\begin{proof}
Since $$\frac{w^2}{v^2}=\frac{1-v}{s}=\frac{(f_++f_-)^2}{4},$$ 
Theorem \ref{th:1-v} together with  Theorem \ref{th:image-Schwarz} yields that
\begin{align*}
\frac{w^2}{v^2}
&=\frac{1}{2}
\Big(\frac{\h_{0,1}(y_1,\tau)\h_{1,0}(y_1,\tau)}
{\h_{0,0}(y_1,\tau)^2}
\frac{\h_{0,0}(y_1,\tau)}
{\h_{1,1}(y_1,\tau)}
+
\frac{\h_{0,1}(y_2,\tau)\h_{1,0}(y_2,\tau)}
{\h_{0,0}(y_2,\tau)^2}
\frac{\i\h_{0,0}(y_2,\tau)}{\h_{1,1}(y_2,\tau)}
\Big)^2\\
&=\frac{1}{2}
\Big(\frac{\h_{0,1}(y_1,\tau)\h_{1,0}(y_1,\tau)}
{\h_{0,0}(y_1,\tau)\h_{1,1}(y_1,\tau)}
+\i
\frac{\h_{0,1}(y_2,\tau)\h_{1,0}(y_2,\tau)}
{\h_{0,0}(y_2,\tau)\h_{1,1}(y_2,\tau)}
\Big)^2,\\
\frac{w}{v}&=\frac{\pm1}{\sqrt{2}}
\Big(\frac{\h_{0,1}(y_1,\tau)\h_{1,0}(y_1,\tau)}
{\h_{0,0}(y_1,\tau)\h_{1,1}(y_1,\tau)}
+\i
\frac{\h_{0,1}(y_2,\tau)\h_{1,0}(y_2,\tau)}
{\h_{0,0}(y_2,\tau)\h_{1,1}(y_2,\tau)}
\Big). 
\end{align*}
To determine the sign for $w/v$, we consider 
multi-valued functions $\sqrt{1-v}$ and $1/\sqrt{s}$ on $C_z$.
We assign branches of these functions so that 
they become positive real  for $P_v$ with $v\in (0,1)$.
Since $\ds{\lim_{P_v\to P_0,v\in(0,1)}\sqrt{1-v}=1}$,
this function admits an expression 
$$\sqrt{1-v}=
\frac{\h_{0,0}(\tau)^2}{2\h_{0,1}(\tau)\h_{1,0}(\tau)}
\Big(\frac{\h_{0,1}(y_1,\tau)\h_{1,0}(y_1,\tau)}
{\h_{0,0}(y_1,\tau)^2}
+
\frac{\h_{0,1}(y_2,\tau)\h_{1,0}(y_2,\tau)}
{\h_{0,0}(y_2,\tau)^2}
\Big)
$$
for $P_v$ with $v\in (0,1)$ by Theorem \ref{th:1-v}.
By Lemma \ref{lem:f+-}, we have 
$$\frac{1}{\sqrt{s}}=\pm \frac{\sqrt{2}\h_{0,1}(\tau)\h_{1,0}(\tau)}
{\h_{0,0}(\tau)^2}\cdot 
\frac{\h_{0,0}(y_1,\tau)}{\h_{1,1}(y_1,\tau)}.
$$
By substituting $P=P_{v_-}$ into this equality, we determine its sign. 
Since $v_-\in (0,1)$, $1/\sqrt{s}$ is positive for $P_{v_-}$. 
On the other hand, we have
$$y_1=\int_{P_0}^{P_{v_-}} 2\f_1=\frac{1}{2},\quad 
\frac{\h_{0,0}(1/2,\tau)}{\h_{1,1}(1/2,\tau)}=
-\frac{\h_{0,1}(0,\tau)}{\h_{1,0}(0,\tau)}
$$
by the third of \eqref{eq:basic-theta}.
Since $\h_{k,l}(\tau)$ are positive real for any pure imaginary $\tau$, 
it turns out that
$$\frac{1}{\sqrt{s}}=-\frac{\sqrt{2}\h_{0,1}(\tau)\h_{1,0}(\tau)}
{\h_{0,0}(\tau)^2}\cdot 
\frac{\h_{0,0}(y_1,\tau)}{\h_{1,1}(y_1,\tau)}
=-\i \frac{\sqrt{2}\h_{0,1}(\tau)\h_{1,0}(\tau)}{\h_{0,0}(\tau)^2}
\cdot \frac{\h_{0,0}(y_2,\tau)}{\h_{1,1}(y_2,\tau)}.
$$
The expressions for $\sqrt{1-v}$ and $1/\sqrt{s}$
together with 
$$
\frac{w}{v}=v^{-1/4}(1-v)^{1/4}(1-vz)^{1/4}
=\sqrt{1-v}\cdot \frac{1}{\sqrt{s}} \quad (v\in (0,1))
$$
yield the expression for $w/v$.
\end{proof}
\end{appendix}

\end{document}